\documentclass[10pt]{amsart}
\usepackage{amsfonts, mathrsfs}
\usepackage{ifthen}
\usepackage{amsthm}
\usepackage{amsmath}
\usepackage{graphicx}
\usepackage{amscd,amssymb,amsthm}
\usepackage{graphicx}
\usepackage{epstopdf}
\usepackage{hyperref}
\usepackage{algorithm2e}
\usepackage{mathtools}
\usepackage{color}

\usepackage[thinc]{esdiff}
\newcounter{minutes}
\divide\time by 60
\newcounter{hours}
\multiply\time by 60 \addtocounter{minutes}{-\time}
\DeclareMathOperator{\arcsec}{arcsec}

\newtheorem{lemma}{Lemma}[section]
\newtheorem{theorem}{Theorem}[section]

\newtheorem{remark}{Remark}[section]

\newcommand{\airya}{\operatorname{Ai}}
\newcommand{\airyb}{\operatorname{Bi}}


\begin{document}
	\title[Asymptotic Behavior of zeros of Bessel function derivatives]{Asymptotic Behavior of zeros of Bessel function derivatives}
	\allowdisplaybreaks
	\author[\'A. Baricz]{\'Arp\'ad Baricz}
	\address{Department of Economics,  Babe\c{s}-Bolyai University, Cluj-Napoca 400591, Romania}
	\address{Institute of Applied Mathematics, \'Obuda University, 1034 Budapest, Hungary}
	\email{bariczocsi@yahoo.com}
	
	\author[P. Kumar]{Pranav Kumar}
	\address{Department of Mathematics,
		Indian Institute of Technology Madras, Chennai 600036, India}
	\email{pranavarajchauhan@gmail.com}
	
	\author[S. Ponnusamy]{Saminathan Ponnusamy}
	\address{Department of Mathematics,
		Indian Institute of Technology Madras, Chennai 600036, India}
	\email{samy@iitm.ac.in}
	
	\newcommand{\ec}{\operatorname{erfc}}
	
	\def\thefootnote{}
	\footnotetext{ \texttt{File:~\jobname .tex,
			printed: \number\year-\number\month-\number\day,
			\thehours.\ifnum\theminutes<10{0}\fi\theminutes}
	} \makeatletter\def\thefootnote{\@arabic\c@footnote}\makeatother
	
	\maketitle
	
\begin{abstract}		
We derive two distinct asymptotic expansions for the zeros $j_{\nu,k}^{(n)}$ of the $n$-th derivative of Bessel function $J_\nu^{(n)}(x)$. The first is a McMahon-type expansion for the case when $k \to \infty$ with fixed $\nu$, for which we also establish an explicit error bound. The second addresses the case when $\nu \to \infty$ with fixed $k$ and it involves the zeros of Airy functions and their derivatives. These results extend and refine the classical work of Wong, Lang, and Olver on the zeros of Bessel functions. In the course of obtaining our main results, we also generalize several auxiliary results, which in turn provide a broader framework for the study of zeros of special functions.
\end{abstract}

\section{\bf Introduction}

The zeros of Bessel functions play a significant role in various problems across applied mathematics and mathematical physics. Their significance is apparent in areas such as quantum mechanics, scattering theory, wave propagation, and related fields (see \cite[Secs. 10.72 and 10.73] {OLBC10},  \cite{DYL6,ELR93,FS08,LZ07,Pa72} and the references therein). Because of their wide-ranging applicability, considerable attention has been devoted to understanding the asymptotic behavior of Bessel functions and their zeros. Such asymptotic information is useful for estimating the zeros of Bessel functions for large parameters, which frequently arise in both theoretical and practical contexts.
	
One of the earliest systematic studies of the asymptotic expansion of Bessel functions' zeros was carried out by the Irish mathematician James McMahon \cite{Mc1895}, who derived fundamental expansions and demonstrated their important applications in physics. In particular, if $j_{\nu,k}$ denote the $k$-th positive zero of Bessel functions $J_\nu(x)$, then for large $k$ and fixed order $\nu > 0$, McMahon obtained the classical asymptotic formula
\begin{equation}\label{Nemes_j}
j_{\nu,k}\sim\left(k+\frac{1}{2}\nu-\frac{1}{4}\right)\pi+\sum_{s=0}^{\infty}\frac{c_s(\nu)}{\left[\left(k+\frac{1}{2}\nu-\frac{1}{4}\right)\pi\right]^{2s+1}} \quad (k\to \infty),
\end{equation}
where the coefficients $c_s(\nu)$ are polynomials in $\nu$. Extending the work of Schafheitlin, Watson \cite[Sec. 15.33]{Wa44} established a lower bound for the positive zeros of $J_\nu(z)$
\begin{align*}
\left(k+\frac{1}{2}\nu-\frac{1}{4}\right)\pi<j_{\nu,k}
\end{align*}
for $k\in\mathbb{N}$ and $-\frac{1}{2}<\nu<\frac{1}{2}$. This estimate was later sharpened by F\"orster and Petras \cite{FP93}. Recently, Nemes \cite{Ne21} settled two longstanding conjectures regarding the enveloping properties of real zeros of cylinder and Airy functions, originally posed by Elbert and Laforgia \cite{EL01} and by Fabijonas and Olver \cite{FO99}, respectively. Nemes further proved that $c_s(\nu)$, in \eqref{Nemes_j}, is indeed a polynomial in $\nu$ of degree $2s$.
	
Now, let us discuss the asymptotic expansion of $j_{\nu,k}$ where $\nu \to \infty$ with $k$ fixed. Let $a_k$ denote the $k$-th negative zero of the Airy function $\airya(x)$, ordered so that $|a_1| < |a_2| < {\ldots}.$ Then, as given in \cite[eq.~10.21.32]{OLBC10},
\begin{align*}
j_{\nu,k}\sim\nu\sum_{x=0}^{\infty}\frac{\alpha_s}{\nu^{\frac{2\nu}{3}}}\quad(\nu\to\infty),
\end{align*}
where each coefficient $\alpha_s$ is a polynomial in $a_k$. This result arises naturally from Olver's \cite[Sec. 7]{Ol54} uniform asymptotic expansion of  $j_{\nu,k}$ as $\nu \to \infty$  \cite{Ol54}, which provides a systematic framework to use the expansion for Bessel differential equation to derive uniform asymptotic expansions of $J_\nu(x)$ in terms of Airy functions and their derivatives, and from these obtained
\begin{align*}
j_{\nu,k}\sim\nu\sum_{s=0}^{\infty}\frac{z_{k,s}}{\nu^{2s}}\quad (\nu\to\infty),
\end{align*}
where $z_{k,s}$ are given implicitly for $s>1$. Later, Wong and Lang \cite{WL90} extended these ideas to obtain asymptotic expansions for the zeros of $J_\nu''(x)$ in two cases: $\nu \to \infty$ with fixed $k$, and $k \to \infty$ with fixed $\nu$. Their method combined with the application of Bessel differential equation by using asymptotic expansions of $J_\nu(x)$ and its derivatives, ultimately reduces the analysis to Olver's turning point framework.

These advances in asymptotic theory are deeply connected with the classical study of the zeros of Bessel functions - a subject that drew the attention of eminent mathematicians such as Bessel, Euler, Fourier, Lommel, Rayleigh, and Stokes. Their pioneering investigations laid the foundation for much of the modern theory, for an extensive historical account, see \cite{Ke14} and the references therein.
	
In recent years, this area continues to develop, with new insights emerging from the interplay of classical analysis, asymptotic methods, and modern function theory. The first author together with Kokologiannaki and Pog\'any \cite{BKP18}, proved that for $\nu > n-1$, all the zeros of the $n$-th derivative $J_\nu^{(n)}(x)$ are real and positive, and moreover, that the zeros of $J_\nu^{(n)}(x)$ and $J_\nu^{(n+1)}(x)$ interlace whenever $\nu \geq n$. Subsequently, Frantzis et. al. \cite{FKP24} investigated the monotonicity of the zeros $j_{\nu,k}^{(n)}$ with respect to $\nu$, providing further structural understanding. Most recently, Dimitrov and Lun \cite{DL25} employed Mittag-Leffler expansions of $J_\nu^{(n)}(x)$ and Jensen's characterization of entire functions in the Laguerre-P\'olya class to analyze the monotonicity and distribution of these zeros. Study of the zeros of other special functions, like Wright, hyper-Bessel functions etc., are also studied in details (see \cite{BS18} and references therein).
	
Despite significant progress in understanding the properties of the zeros of the $n$-th order derivative of Bessel functions, the asymptotic expansions of these zeros have not been explored in the literature. This paper aims to fill this gap by presenting McMahon-type asymptotic expansions and by analyzing the case where $\nu \to \infty$ with fixed $n$. We employ the standard approach introduced by McMahon \cite{Mc1895} to derive the asymptotic expansion of $J_\nu^{(n)}(x)$ and also highlight an alternate method to obtain similar results. Subsequently, by using the properties of the zeros of the $n$-th derivative of Bessel functions \cite{BKP18}, denoted by $j_{\nu,k}^{(n)}$, we establish the asymptotic behavior of $j_{\nu,k}^{(n)}$ as $k \to \infty$.
For the second case, namely, $\nu \to \infty$ with fixed $k$, we adapt the method of Wong and Lang \cite{WL91}, with suitable modifications, in combination with Hethcote’s theorem \cite[Theorem 1]{He70}, to obtain an initial approximation of $j_{\nu,k}^{(n)}$ for large $\nu$. Finally, we employ the Bessel differential equation together with Olver’s method to derive the full asymptotic expansion of $j_{\nu,k}^{(n)}$ for $\nu \to \infty$ and finite $k$. Throughout this paper, if not stated otherwise, empty sums are taken to be zero. Additionally, $\mathbb{N}$ is the set of all positive integers and $\mathbb{N}_0= \mathbb{N}\cup\{0\}$.
	
\section{\bf McMahon-type expansion for zeros of Bessel functions derivatives}	
\subsection{Some initial results}
In this section we will present some basic results on expressing even and odd derivatives of Bessel functions, for large $x$, in terms of $sine$ and $cosine$ functions. The method of the proof is similar to that used by McMahon \cite{Mc1895} along with the use of mathematical induction. It is interesting to note that similar results can also be obtained by expressing the $n$-th order derivatives of the Bessel function in terms of Bessel function and its derivatives (see \ref{J^n=J+J'}) and then substituting the expansion for Bessel function or its first derivative from \cite{Mc1895}.

\begin{theorem}\label{Theorem1}
For $n\in\mathbb{N}_0$ and large $x$, the $2n$-th derivative of the Bessel function of first kind $J_\nu(x)$, can be expressed as
\begin{equation}\label{2n_der}
\begin{split}
	\sqrt{\frac{1}{2}\pi x}J_\nu^{(2n)}(x)&=\cos \left(x-\frac{\nu\pi}{2}-\frac{\pi}{4}\right)\tau_\nu^{(2n)}(x)+
	\sin\left(x-\frac{\nu\pi}{2}-\frac{\pi}{4}\right)\theta_\nu^{(2n)}(x),
\end{split}
\end{equation}
where
$$\tau_\nu^{(2n)}(x)=\sum_{m=0}^{\infty}\frac{\alpha_{2m,2n}}{x^{2m}},\ \theta_\nu^{(2n)}(x)=\sum_{m=0}^{\infty}\frac{\alpha_{2m+1,2n}}{x^{2m+1}},\ \alpha_{2m,2n}=\alpha_{2m,2n-1}-\frac{4m-1}{2}\alpha_{2m-1,2n-1}$$
and
\begin{align*}\label{}
\alpha_{2m+1,2n}=-\alpha_{2m-1,2n-1}-\frac{4m+1}{2}\alpha_{2m,2n-1}.
\end{align*}
Moreover, the $(2n+1)$-th derivative of $J_\nu(x)$ can be expressed as
\begin{equation}\label{2n+1_der}
\begin{split}
	\sqrt{\frac{1}{2}\pi x}J_\nu^{(2n+1)}(x)&=\cos \left(x-\frac{\nu\pi}{2}-\frac{\pi}{4}\right)\theta_\nu^{(2n+1)}(x)+
	\sin\left(x-\frac{\nu\pi}{2}-\frac{\pi}{4}\right)\tau_\nu^{(2n+1)}(x),
\end{split}
\end{equation}
where
$$\tau_\nu^{(2n+1)}(x)=\sum_{m=0}^{\infty}\frac{\alpha_{2m,2n+1}}{x^{2m}},\ \theta_\nu^{(2n+1)}(x)=\sum_{m=0}^{\infty}\frac{\alpha_{2m+1,2n+1)}}{x^{2m+1}},$$
\begin{equation}\label{rec_rln_odd}
\alpha_{2m,2n+1}=-\frac{4m-1}{2}\alpha_{2m-1,2n}- \alpha_{2m,2n}
\end{equation}
and
\begin{equation}\label{rec_rln_odd1}
\alpha_{2m+1,2n+1}=\alpha_{2m+1,2n}-\frac{4m+1}{2}\alpha_{2m,2n}.
\end{equation}
In particular for $n=0$, we obtain that
$$\alpha_{2m+1,0}=(-1)^{m+1}A_{2m+1}(\nu)\quad \mbox{and} \quad \alpha_{2m,0}=(-1)^mA_{2m}(\nu)$$
with
$$A_s(\nu)=\frac{(4\nu^2-1)(4\nu^2-3^2)\cdots\left(4\nu^2-(2s-1)^2\right)}{s!8^s}.$$
\end{theorem}
	
\begin{remark}\label{remark1}
{\em We can write \eqref{2n_der} and \eqref{2n+1_der} for large $x$ as
$$\sqrt{\frac{1}{2}\pi x}J_\nu^{(2n)}(x)=\alpha_{0,2n}\cos \left(x-\frac{\nu\pi}{2}-\frac{\pi}{4}\right)+\mathcal{O}\left(\frac{1}{x}\right)$$
and
$$\sqrt{\frac{1}{2}\pi x}J_\nu^{(2n+1)}(x)=\alpha_{0,2n+1}\sin\left(x-\frac{\nu\pi}{2}-\frac{\pi}{4}\right)+\mathcal{O}\left(\frac{1}{x}\right),$$
respectively. These expressions represent the oscillatory nature of the higher order derivatives of Bessel functions of first kind. These results also align with the asymptotic form of $n$th derivative of Bessel functions obtained by E.A. Skelton \cite[Eq. (1)]{Sk02}.}
\end{remark}
	
\begin{proof}[\bf Proof of Theorem \ref{Theorem1}]
Recall that the asymptotic relations for Bessel functions and their first derivatives \cite[Eqs. (2.4) and (2.5)]{WL90}, can be expressed, respectively, as
\begin{equation}\label{Bsl_asym}
\sqrt{\frac{1}{2}\pi x}J_\nu(x)=\cos(x-\alpha)\tau_\nu^{(0)}(x)+\sin(x-\alpha)\theta_\nu^{(0)}(x)
\end{equation}
and
\begin{equation}\label{Bsl_der_asym}
\sqrt{\frac{1}{2}\pi x}J_\nu^{(1)}(x)=\cos(x-\alpha)\theta_\nu^{(1)}(x)+\sin(x-\alpha)\tau_\nu^{(1)}(x),
\end{equation}
where
$$\tau_\nu^{(0)}(x)\sim\sum_{m=0}^{\infty}\frac{\alpha_{2m,0}}{x^{2m}},\ \theta_\nu^{(0)}(x)\sim\sum_{m=0}^{\infty}\frac{\alpha_{2m+1,0}}{x^{2m+1}},\
\theta_\nu^{(1)}(x)\sim\sum_{m=0}^{\infty}-\frac{4\nu^2+4(2m+1)^2-1}{4\nu^2-(4m+1)^2}\frac{\alpha_{2m,0}}{x^{2m+1}},$$
$$\tau_\nu^{(1)}(x)\sim\sum_{m=0}^{\infty}\frac{4\nu^2+16m^2-1}{4\nu^2-(4s-1)^2}\frac{\alpha_{2m+1,0}}{x^{2m}},\
\alpha_{2m+1,0}=(-1)^{m+1}A_{2m+1}(\nu),\ \alpha_{2m,0}=(-1)^mA_{2m}(\nu).$$
We write the expressions for $\tau_\nu^{(1)}(x)$ and $\theta_\nu^{(1)}(x)$ as
$$\tau_\nu^{(1)}(x)\sim\sum_{m=0}^{\infty}\frac{\alpha_{2m+1,1}}{x^{2m+1}}\quad \mbox{and} \quad
\theta_\nu^{(1)}(x)\sim\sum_{m=0}^{\infty}\frac{\alpha_{2m,1}}{x^{2m}}.$$
Notice that for $n=0$, the asymptotic expansion of Bessel functions and its first derivative can be derived from \eqref{2n_der} and \eqref{2n+1_der}, respectively, which align with \eqref{Bsl_asym} and \eqref{Bsl_der_asym}. We will use mathematical induction on $n$ to prove the assertion of the theorem. Let us assume that \eqref{2n_der} and \eqref{2n+1_der} hold up to the $2n$-th derivative of the Bessel functions, for some positive integer $n$. Then in view of equation \eqref{2n_der} we can write that
\begin{equation}\label{2m_der}
\sqrt{\frac{1}{2}\pi x}J_\nu^{(2n)}(x)=\cos(x-\alpha)\tau_\nu^{(2n)}(x)+\sin(x-\alpha)\theta_\nu^{(2n)}(x),
\end{equation}
where
$$\tau_\nu^{(2n)}(x)=\sum_{m=0}^{\infty}\frac{\alpha_{2m,2n}}{x^{2m}},\ \theta_\nu^{(2n)}(x)=\sum_{m=0}^{\infty}\frac{\alpha_{2m+1,2n}}{x^{2m+1}},\
\alpha_{2m,2n}=\frac{\alpha_{2m,2n-1}}{2}-\frac{4(n-1)}{2^2}\alpha_{2m-1,2n-1},$$
and
$$\alpha_{2m+1,2n}=-\frac{\alpha_{2m-1,2n-1}}{2}-\frac{4n+1}{2^2}\alpha_{2m,2n-1}.$$
Dividing both sides of \eqref{2m_der} by $\sqrt{x}$ we obtain
\begin{align*}
\sqrt{\frac{\pi}{2}}J_\nu^{(2n)}(x)&=\cos(x-\alpha)\frac{\tau_\nu^{(2n)}(x)}{\sqrt{x}}+\sin(x-\alpha)\frac{\theta_\nu^{(2n)}(x)}{\sqrt{x}}\\
&=\cos(x-\alpha)\sum_{m=0}^{\infty}\frac{\alpha_{2m,2n}}{x^{\frac{4m+1}{2}}}+\sin(x-\alpha)\sum_{m=0}^{\infty}\frac{\alpha_{2m+1,2n}}{x^{\frac{4m+3}{2}}}.
\end{align*}
Differentiating both sides of above equation with respect to $x$, we arrive at
\begin{align*}
\sqrt{\frac{\pi}{2}}J_\nu^{(2n+1)}(x)=&-\cos(x-\alpha)\sum_{n=0}^{\infty}\frac{(4n+1)}{2}\frac{\alpha_{2m,2n}}{x^{\frac{4n+3}{2}}}
-\sin(x-\alpha)\sum_{n=0}^{\infty}\frac{2(2n+1)+1}{2}\frac{\alpha_{2m+1,2n}}{x^{\frac{4n+5}{2}}}\\
&-\sin(x-\alpha)\sum_{m=0}^{\infty}\frac{\alpha_{2m,2n}}{x^{\frac{4n+1}{2}}}+\cos(x-\alpha)\sum_{m=0}^{\infty}\frac{\alpha_{2m+1,2n}}{x^{\frac{4n+3}{2}}}.
\end{align*}
Now, multiplying both sides of the above equation by $\sqrt{x}$, we obtain
\begin{align*}
\sqrt{\frac{\pi x}{2}}J_\nu^{(2n+1)}(x)&=-\cos(x-\alpha)\sum_{m=0}^{\infty}\frac{(4m+1)}{2}\frac{\alpha_{2,2n}}{x^{2m+1}}-\sin(x-\alpha)\sum_{m=0}^{\infty}\frac{2(2m+1)+1}{2}\frac{\alpha_{2m+1}^{(2n)}}{x^{2n+2}}\\
&\hspace*{2em}-\sin(x-\alpha)\sum_{m=0}^{\infty}\frac{\alpha_{2m,2n}}{x^{2n}}+\cos(x-\alpha)\sum_{m=0}^{\infty}\frac{\alpha_{2m+1,2n}}{x^{2m+1}}\\
&=\sin(x-\alpha)\left[- \sum_{m=0}^{\infty}\frac{2(2n+1)+1}{2}\frac{\alpha_{2m+1,2n}}{x^{2m+2}}- \sum_{m=0}^{\infty}\frac{\alpha_{2m,2n}}{x^{2m}} \right]\\
&\hspace*{2em}+\cos(x-\alpha)\left[\sum_{m=0}^{\infty}\frac{\alpha_{2m+1,2n}}{x^{2m+1}} -\sum_{m=0}^{\infty}\frac{(4m+1)}{2}\frac{\alpha_{2m,2n}}{x^{2m+1}} \right]\\
&=\sin(x-\alpha)\left[- \sum_{m=0}^{\infty}\frac{4m-1}{2}\frac{\alpha_{2m-1,2n}}{x^{2m}}- \sum_{m=0}^{\infty}\frac{\alpha_{2m,2n}}{x^{2m}} \right]\\
&\hspace*{2em}+\cos(x-\alpha)\left[\sum_{m=0}^{\infty}\frac{\alpha_{2m+1,2n}}{x^{2m+1}} -\sum_{m=0}^{\infty}\frac{(4m+1)}{2}\frac{\alpha_{2m,2n}}{x^{2m+1}} \right]\\
&=\sin(x-\alpha)\left[ \sum_{m=0}^{\infty}\left(-\frac{(4m-1)\alpha_{2m-1,2n}}{2}- \alpha_{2m,2n}\right)\frac{1}{x^{2n}}\right]\\
&\hspace*{2em}+\cos(x-\alpha)\left[\sum_{m=0}^{\infty}\left(\alpha_{2m+1,2n} -\frac{(4m+1)\alpha_{2m,2n}}{2}\right)\frac{1}{x^{2m+1}} \right]\\
&=\sin(x-\alpha)\sum_{m=0}^{\infty}\frac{\alpha_{2m,2n+1}}{x^{2m}}+\cos(x-\alpha)\sum_{m=0}^{\infty}\frac{\alpha_{2m+1,2n+1}}{x^{2m+1 }},
\end{align*}
which satisfies the relations \eqref{rec_rln_odd} and \eqref{rec_rln_odd1}. We now derive the expansion of the even derivatives of Bessel functions by using the expansion of the previous odd derivatives. To accomplish this, suppose that \eqref{2n_der} and \eqref{2n+1_der} hold up to the $(2n+1)$-th derivative of the Bessel functions, for some positive integer $n$. Considering the equation \eqref{2n+1_der} in terms of $n$, we obtain that
\begin{equation}\label{odd_der}
\sqrt{\frac{1}{2}\pi x}J_\nu^{(2n+1)}(x)=\cos(x-\alpha)\theta_\nu^{(2n+1)}(x)+\sin(x-\alpha)\tau_\nu^{(2n+1)}(x),
\end{equation}
where
$$\tau_\nu^{(2n+1)}(x)=\sum_{m=0}^{\infty}\frac{\alpha_{2m,2n+1}}{x^{2m}},\ \theta_\nu^{(2n+1)}(x)=\sum_{m=0}^{\infty}\frac{\alpha_{2m+1,2n+1}}{x^{2m+1}},\
\alpha_{2m,2n+1}=\frac{-(4m-1)\alpha_{2m-1,2n}}{2}- \alpha_{2m,2n},$$
and
$$\alpha_{2m+1,2n}=\alpha_{2m+1,2n} -\frac{(4m+1)\alpha_{2m,2n}}{2}.$$
We will use the same procedure as in the case of the $2n$-th derivative of the Bessel functions to obtain the expression for $(2n+2)$-th derivative of Bessel functions. Dividing \eqref{odd_der} by $\sqrt{x}$ and differentiating with respect to $x$, we obtain
\begin{align*}
\sqrt{\frac{\pi}{2}}J_\nu^{(2n+2)}(x)&=\sin(x-\alpha)\left[\sum_{m=0}^{\infty}-\frac{\alpha_{2m+1,2n+1}}{x^{\frac{4m+3}{2}}}
-\sum_{m=0}^{\infty}\left(\frac{4m+1}{2}\frac{\alpha_{2m,2n+1}}{x^{\frac{4m+3}{2}}}\right)\right]\\
&\hspace*{2em}+\cos(x-\alpha)\left[\sum_{m=0}^{\infty}\frac{\alpha_{2m,2n+1}}{x^{\frac{4m+1}{2}}}
-\sum_{m=0}^{\infty}\left(\frac{4m+3}{2}\frac{\alpha_{2m+1,2n+1}}{x^{\frac{4m+5}{2}}}\right)\right].
\end{align*}
On multiplying both sides by $\sqrt{x}$ and arranging the terms we rewrite the above equation as
\begin{align*}
\sqrt{\frac{\pi x}{2}}J_\nu^{(2n+2)}(x)&=\sin(x-\alpha)\sum_{m=0}^{\infty}\left(-\alpha_{2m+1,2n+1}-\frac{4m+1}{2}\alpha_{2m,2n+1}\right)\frac{1}{x^{2m+1}}\\
&\hspace*{2em}+\cos(x-\alpha)\left[\sum_{m=0}^{\infty}\frac{\alpha_{2m,2n+1}}{x^{2m}}-\sum_{m=0}^{\infty}\left(\frac{4m+3}{2}\frac{\alpha_{2m+1,2n+1}}{x^{2(m+1)}}\right)\right]\\
&=\sin(x-\alpha)\sum_{m=0}^{\infty}\left(-\alpha_{2m+1,2n+1}-\frac{4m+1}{2}\alpha_{2m,2n+1}\right)\frac{1}{x^{2m+1}}\\
&\hspace*{2em}+\cos(x-\alpha)\left[\sum_{m=0}^{\infty}\frac{\alpha_{2m,2n+1}}{x^{2m}}-\sum_{m=1}^{\infty}\left(\frac{4m-1}{2}\frac{\alpha_{2m-1,2n+1}}{x^{2m}}\right)\right]\\
&=\sin(x-\alpha)\sum_{m=0}^{\infty}\frac{\alpha_{2m+1,2n+2}}{x^{2m+1}}+\cos(x-\alpha)\sum_{m=0}^{\infty}\frac{\alpha_{2m,2n+2}}{x^{2m}},
\end{align*}
where $\alpha_{2m+1,2n+2}\text{ and }\alpha_{2m,2n+2}$ are given by
$$\alpha_{2m+1,2n+2}=-\alpha_{2m+1,2n+1}-\frac{4m+1}{2}\alpha_{2m,2n+1}$$
and
$$\alpha_{2m,2n+2}= \alpha_{2m,2n+1}-\frac{4m-1}{2}\alpha_{2m-1,2n+1},$$
respectively. The above results establish the recurrence relation for the coefficients in the expansion of the $2n$-th derivative of the Bessel functions, as stated in Theorem \ref{Theorem1}. This completes the proof.
\end{proof}

In the next theorem, we establish a bound for the error in the asymptotic expansion of the derivatives of Bessel functions obtained in Theorem \ref{Theorem1}. This result generalizes \cite[Eq. (3.6)]{WL91} and is useful in deriving an error bound for the McMahon-type asymptotic expansion of $J_\nu^{(n)}(x)$. Moreover, this theorem may be of independent interest to researchers studying the asymptotic behavior of $J_\nu^{(n)}(x)$ for large $x$.

\begin{theorem}\label{Theorem2}
For $n\in\mathbb{N}_0$, let the even order derivative of the Bessel functions be denoted by
\begin{equation}\label{2n_der_del}
\sqrt{\frac{\pi x}{2}}J^{(2n)}_\nu(x)=(-1)^n\cos\left(x-\frac{1}{2}\nu \pi-\frac{1}{4}\pi\right)+\delta_{2n}(\nu,x),
\end{equation}
where $\delta_{2n}$ consists of terms on the right-hand side of the equation \eqref{2n_der} (cf. Remark \ref{remark1}). Then for $\nu\geq-2n+\frac{1}{2},n\in\mathbb{N}$ and large $x$, the expression $\delta_{2n}(\nu,x)$ is bounded as follows
\begin{equation}\label{2n_del_bnd}
|\delta_{2n}(\nu,x)|\leq \frac{4(\nu+2n)^2-1}{4x}\exp\left\{\frac{4(\nu+2n)^2-1}{4x}\right\}.
\end{equation}
Similarly, let the $(2n+1)$-th derivative of $J_\nu(x)$ be denoted by
\begin{equation}\label{2n+1_der_del}
\sqrt{\frac{\pi x}{2}}J^{(2n+1)}_\nu(x)=(-1)^{n+1}\sin\left(x-\frac{1}{2}\nu \pi-\frac{1}{4}\pi\right)+\delta_{2n+1}(\nu,x),
\end{equation}\label{2n+1_del_bnd}
where for $\nu\geq-2n+\frac{3}{2},n\in\mathbb{N}$ and large $x$, $\delta_{2n+1}$ consists of the remaining terms on the right-hand side of the equation \eqref{2n+1_der} (cf. Remark \ref{remark1}). Then the expression $\delta_{2n+1}(\nu,x)$ is bounded as follows
\begin{equation}
|\delta_{2n+1}(\nu,x)|\leq \frac{4(\nu+2n+1)^2-1}{4x}\exp\left\{\frac{4(\nu+2n+1)^2-1}{4x}\right\}.
\end{equation}
\end{theorem}

\begin{proof}[\bf Proof of Theorem \ref{Theorem2}]\label{Th2_proof}
From \cite[eq. 10.6.7]{OLBC10} we write the $2n$-th derivative of the Bessel functions as
\begin{align*}
J^{(2n)}_\nu(x)=\frac{1}{2^{2n}}\sum_{m=0}^{2n}(-1)^m\binom{2n}{m}J_{\nu-2n+2m}(x).
\end{align*}
Moreover, by using \eqref{Bsl_asym} and the above equation, we obtain that
\begin{align*}
J^{(2n)}_\nu(x)&=\frac{1}{2^{2n}}\sum_{m=0}^{2n}(-1)^m\binom{2n}{m}\left[\cos\left(x-\frac{1}{2}(\nu-2n+2m) \pi-\frac{1}{4}\pi\right)+\delta_{}(\nu-2n+2m,x)\right]\\
&=\frac{1}{2^{2n}}\sum_{m=0}^{2n}(-1)^m\binom{2n}{m}\cos\left(x-\frac{1}{2}(\nu-2n+2m) \pi-\frac{1}{4}\pi\right)\\
&\quad+\frac{1}{2^{2n}}\sum_{m=0}^{2n}(-1)^m\binom{2n}{m}\delta_{}(\nu-2n+2m,x).
\end{align*}
Let us consider the first term on the right hand side of the above equation, which we can rewrite as
\begin{align*}
\frac{1}{2^{2n}}&\sum_{m=0}^{2n}(-1)^m\binom{2n}{m}\cos\left(x-\frac{1}{2}(\nu-2n+2m) \pi-\frac{1}{4}\pi\right)\\
&=\frac{1}{2^{2n}}\sum_{m=0}^{2n}(-1)^m\binom{2n}{m}\cos\left(x-\frac{1}{2}\nu\pi+(n-m) \pi-\frac{1}{4}\pi\right)\\
&=\frac{1}{2^{2n}}\sum_{m=0}^{2n}(-1)^m\binom{2n}{m}(-1)^{(n-m)}\cos\left(x-\frac{1}{2}\nu\pi-\frac{1}{4}\pi\right)\\
&=\frac{(-1)^n}{2^{2n}}\cos\left(x-\frac{1}{2}\nu\pi-\frac{1}{4}\pi\right)\sum_{m=0}^{2n}\binom{2n}{m}\\&=(-1)^n\cos\left(x-\frac{1}{2}\nu\pi-\frac{1}{4}\pi\right).
\end{align*}
Now, we denote the second term on the right hand side of the above equation as
\begin{equation}\label{delta_even_der}
\delta_{2n}=\frac{1}{2^{2n}}\sum_{m=0}^{2n}(-1)^m\binom{2n}{m}\delta_{}(\nu-2n+2m,x).
\end{equation}
To show that $\delta_{2n}$ is bounded as \eqref{2n_del_bnd}, we use the following result on Bessel functions
\begin{align*}
J_\nu(x)=\sqrt{\frac{2}{\pi x}}\left[\cos\left(x-\frac{1}{2}\nu \pi-\frac{1}{4}\pi\right)+\delta_1(\nu,x)\right],
\end{align*}
where
\begin{equation}\label{del_besel} |\delta_1(\nu,x)|\leq\frac{4\nu^2-1}{4x}\exp\left\{\frac{4\nu^2-1}{4x}\right\}\quad \mbox{for}\ \nu\geq\frac{1}{2},
\end{equation}
which can be proved by using the corresponding results on Bessel functions of the third kind $H_\nu^{(1)}(x)$ and $H_\nu^{(2)}(x)$ \cite[p. 266]{Ol74} and the relation $J_\nu(x)=\frac{1}{2}\left[H_\nu^{(1)}(x)+H_\nu^{(2)}(x)\right]$ (cf. \cite[p. 512]{WL90}). Notice that $\delta_{2n}(\nu-2n+2m,x)$ on the right hand side of the equation \eqref{delta_even_der} is the error corresponding to the asymptotic approximation of $J_{\nu-2n+2m}(x)$. By using the triangle inequality, the equation \eqref{del_besel} and the summation of the binomial expansion, we obtain that
\begin{align*}
|\delta_{2n}|&=\frac{1}{2^{2n}}\left|\sum_{m=0}^{2n}(-1)^m\binom{2n}{m}\delta_{2n}(\nu-2n+2m,x)\right|\\
&\leq\frac{1}{2^{2n}}\sum_{m=0}^{2n}\binom{2n}{m}\left|\delta_{2n}(\nu-2n+2m,x)\right|\\
&\leq\frac{1}{2^{2n}}\sum_{m=0}^{2n}\binom{2n}{m}\frac{4(\nu-2n+2m)^2-1}{4x}\exp\left\{\frac{4(\nu-2n+2m)^2-1}{4x}\right\}\\
&\leq\frac{1}{2^{2n}}\frac{4(\nu+2n)^2-1}{4x}\exp\left\{\frac{4(\nu+2n)^2-1}{4x}\right\}\sum_{m=0}^{2n}\binom{2n}{m}\\
&=\frac{4(\nu+2n)^2-1}{4x}\exp\left\{\frac{4(\nu+2n)^2-1}{4x}\right\}.
\end{align*}
The above inequality holds for $\nu-2n+2m\geq \frac{1}{2}$  for different $\delta_{2n}(\nu-2n+2m,x)$, and the final inequality holds for
$\nu \geq -2n+\frac{1}{2}.$ From \cite[Eq. 10.6.7]{OLBC10} we write the $(2n+1)$-th derivative of the Bessel functions as
\begin{align*}
J^{(2n+1)}_\nu(x)=\frac{1}{2^{2n+1}}\sum_{m=0}^{2n+1}(-1)^m\binom{2n+1}{m}J_{\nu-2n-1+2m}(x),
\end{align*}
that is
\begin{align*}
J^{(2n+1)}_\nu(x)&=\frac{1}{2^{2n+1}}\sum_{m=0}^{2n+1}(-1)^m\binom{2n+1}{m}\left[\cos\left(x-\frac{1}{2}(\nu-2n-1+2m) \pi-\frac{1}{4}\pi\right) \right. \\
& \left. \hspace*{2em}+\delta_{}(\nu-2n-1+2m,x)\right]\\
&=\frac{1}{2^{2n+1}}\sum_{m=0}^{2n+1}(-1)^m\binom{2n+1}{m}\cos\left(x-\frac{1}{2}(\nu-2n-1+2m) \pi-\frac{1}{4}\pi\right)\\
&\quad+\frac{1}{2^{2n+1}}\sum_{m=0}^{2n+1}(-1)^m\binom{2n+1}{m}\delta_{}(\nu-2n-1+2m,x).
\end{align*}
Now, let us consider the first term on the right hand side of the above equation, which can be expressed as
\begin{align*}
&\frac{1}{2^{2n+1}}\sum_{m=0}^{2n+1}(-1)^m\binom{2n+1}{m}\cos\left(x-\frac{1}{2}(\nu-2n-1+2m) \pi-\frac{1}{4}\pi\right)\\
&=\frac{1}{2^{2n+1}}\sum_{m=0}^{2n+1}(-1)^m\binom{2n+1}{m}\cos\left(x-\frac{1}{2}\nu\pi+(n-m) \pi+\frac{\pi}{2}-\frac{1}{4}\pi\right)\\
&=\frac{1}{2^{2n+1}}\sum_{m=0}^{2n+1}(-1)^m\binom{2n+1}{m}(-1)^{(n-m+1)}\sin\left(x-\frac{1}{2}\nu\pi-\frac{1}{4}\pi\right)\\
&=\frac{(-1)^{n+1}}{2^{2n+1}}\sin\left(x-\frac{1}{2}\nu\pi-\frac{1}{4}\pi\right)\sum_{m=0}^{2n+1}\binom{2n+1}{m}\\
&=(-1)^{n+1}\sin\left(x-\frac{1}{2}\nu\pi-\frac{1}{4}\pi\right).
\end{align*}
Now, we denote the second term on the right hand side of the above equation as
\begin{equation}\label{delta_odd_der}
\delta_{2n+1}=\frac{1}{2^{2n+1}}\sum_{m=0}^{2n+1}(-1)^m\binom{2n+1}{m}\delta_{}(\nu-2n-1+2m,x).
\end{equation}
Notice that $\delta_{2n}(\nu-2n-1+2m,x)$ on the right hand side of the equation \eqref{delta_odd_der} is the error corresponding to the asymptotic approximation of $J_{\nu-2n-1+2m}(x)$. By using the triangle inequality, the equation \eqref{del_besel} and the summation of binomial expansion, we obtain that
\begin{align*}
|\delta_{2n+1}|&=\frac{1}{2^{2n+1}}\left|\sum_{m=0}^{2n+1}(-1)^m\binom{2n+1}{m}\delta_{2n+1}(\nu-2n-1+2m,x)\right|\\
&\leq\frac{1}{2^{2n+1}}\sum_{m=0}^{2n+1}\binom{2n+1}{m}\left|\delta_{2n+1}(\nu-2n+2m-1,x)\right|\\
&\leq\frac{1}{2^{2n+1}}\sum_{m=0}^{2n+1}\binom{2n+1}{m}\frac{4(\nu-2n+2m-1)^2-1}{4x}\exp\left\{\frac{4(\nu-2n+2m-1)^2-1}{4x}\right\}\\
&\leq\frac{1}{2^{2n+1}}\frac{4(\nu+2n+1)^2-1}{4x}\exp\left\{\frac{4(\nu+2n+1)^2-1}{4x}\right\}\sum_{m=0}^{2n+1}\binom{2n+1}{m}\\
&=\frac{4(\nu+2n+1)^2-1}{4x}\exp\left\{\frac{4(\nu+2n+1)^2-1}{4x}\right\},
\end{align*}
for $\nu\geq-2n+\frac{3}{2}$.
\end{proof}

Before finding the McMahon asymptotic expansion of Bessel functions derivatives, let us review some related results. Baricz et. al. \cite{BKP18} proved that, for $\nu>n-1$, all zeros of $J_\nu^{(n)}(x)$ are real and simple. They also conjectured that, for every $n\in\mathbb{N}$ , the positive zeros of $J_\nu^{(k)}(x)$ are increasing functions of the parameter $\nu$, for $\nu\in\left(n-1,\infty\right)$. This conjecture was recently settled  by Dimitrov and Lun \cite{DL25}.

\subsection{McMahon-type expansion for $j_{\nu, k}^{(n)}$}
In \cite[p. 247]{Ol74}, Olver used the inversion technique to derive the McMahon expansion for the zeros of Bessel functions of the first kind $J_\nu(z)$. Wong and Lang \cite{WL90} used the same technique to study the zeros of $J_\nu^{\prime\prime}(x)$. Here we will use the argument of Olver to study the odd as well as even order derivatives of Bessel functions, separately. Lastly we will outline that the method used by McMahon is also applicable to derive the asymptotic expansion of zeros of Bessel functions.

First, we will use the argument of Olver \cite{Ol74} for odd order derivatives of Bessel functions. By using \eqref{2n+1_der}, we write
\begin{align*}
\sqrt{\frac{1}{2}\pi x}J_\nu^{2m+1}(x)=\cos\left(x-\frac{\nu \pi}{2}-\frac{\pi}{4}\right)\theta_\nu^{(2m+1)}(x)+
\sin\left(x-\frac{\nu \pi}{2}-\frac{\pi}{4}\right)\tau_\nu^{(2m+1)}(x),
\end{align*}
where $$\theta_\nu^{(2m+1)}(x)=\sum_{n=0}^{\infty}\frac{\alpha_{2n+1,2m+1}}{x^{2n+1}}=\frac{\alpha_{1,2m+1}}{x}+\frac{\alpha_{3,2m+1}}{x^3}+\frac{\alpha_{5,2m+1}}{x^5}+\ldots$$ and $$\tau_\nu^{(2m+1)}(x)=\sum_{n=0}^{\infty}\frac{\alpha_{2n,2m+1}}{x^{2n}}=\alpha_{0,2m+1}+\frac{\alpha_{2,2m+1}}{x^2}+\frac{\alpha_{4,2m+1}}{x^4}+{\ldots}.$$
If $x$ is a zero of $J_\nu^{2m+1}(x)$, then considering the series expansion of  $\theta_\nu^{(2m+1)}(x)\text{ and }\tau_\nu^{(2m+1)}(x)$ and the fact that $|\cos x|\leq1$, the first approximation is given by
\begin{equation}\label{sin_eqn}
\sin\left(x-\frac{\nu \pi}{2}-\frac{\pi}{4}\right)+\mathcal{O}\left(\frac{1}{x}\right)=0.
\end{equation}
When $x$ is large, the left-hand side is dominated by the first term. The above equation implies
$$x\sim (k-1)\pi+\frac{\nu \pi}{2}+\frac{\pi}{4},$$
for some large integer $k$. This is the first approximation to the root of the equation \eqref{sin_eqn}. In view of the fact that the large $x$ is equivalent to the integer $k$, the equation \eqref{sin_eqn} we write
\begin{align*}
x&=(k-1)\pi+\frac{\nu \pi}{2}+\frac{\pi}{4}+\mathcal{O}\left(\frac{1}{x}\right)\\
&=k\pi+\frac{\nu \pi}{2}-\frac{3\pi}{4}+\mathcal{O}\left(\frac{1}{k}\right)\\
&=\beta+\mathcal{O}\left(\frac{1}{k}\right),
\end{align*}
where $\beta=k\pi+\frac{1}{2}\nu \pi-\frac{3\pi}{4}$. Now we write $$\sin\left(x-\frac{1}{2}\nu \pi-\frac{\pi}{4}\right)=\sin\left(x-\beta+\beta-\frac{1}{2}\nu \pi-\frac{\pi}{4}\right)=\sin\left(x-\beta+(k-1)\pi\right),$$
which can be written as $\sin\left(x-\frac{\nu \pi}{2}-\frac{\pi}{4}\right)=(-1)^{(k-1)}\sin\left(x-\beta\right)$. Moreover, note that $$\cos\left(x-\frac{1}{2}\nu \pi-\frac{\pi}{4}\right)=\cos\left(x-\beta+(k-1)\pi\right)=(-1)^{k-1}\cos\left(x-\beta\right).$$
Consequently,
\begin{align*}
\cos\left(x-\frac{\nu \pi}{2}-\frac{\pi}{4}\right)\theta_\nu^{(2m+1)}(x)+\sin\left(x-\frac{\nu \pi}{2}-\frac{\pi}{4}\right)\tau_\nu^{(2m+1)}(x)=0
\end{align*}
implies that
\begin{align*}
\tan(x-\beta)=-\frac{\theta_\nu^{(2m+1)}(x)}{\tau_\nu^{(2m+1)}(x)}=-\left.\sum_{n=0}^{\infty}\frac{\alpha_{2n+1,2m+1}}{x^{2n+1}}\right/\sum_{n=0}^{\infty}\frac{\alpha_{2n,2m+1}}{x^{2n}}
\end{align*}
or equivalently
\begin{align*}
x-\beta&=-\arctan\left[\left.\sum_{n=0}^{\infty}\frac{\alpha_{2n+1,2m+1}}{x^{2n+1}}\right/\sum_{n=0}^{\infty}\frac{\alpha_{2n,2m+1}}{x^{2n}}\right].
\end{align*}
By using the expansion of $\arctan(x)$ and the binomial expansion we obtain that
\begin{align*}
x=\beta-\frac{\alpha_{1,2m+1}}{\alpha_{0,2m+1}}\frac{1}{x}-\left(\frac{\alpha_{3,2m+1}}{\alpha_{0,2m+1}}-
\frac{\alpha_{1,2m+1}\alpha_{2,2m+1}}{(\alpha_{0,2m+1})^2}-\frac{1}{3}\frac{(\alpha_{1,2m+1})^3}{(\alpha_{0,2m+1)})^3}\right)\frac{1}{x^3}-{\ldots}.
\end{align*}
Substituting $x=\beta+\mathcal{O}\left(\frac{1}{k}\right)$ in the last equation, yields
\begin{align*}
x=\beta-\frac{\alpha_{1,2m+1}}{\alpha_{0,2m+1)}}\frac{1}{\beta}+\mathcal{O}\left(\frac{1}{\beta^3}\right).
\end{align*}
Further substitution gives
\begin{equation}\label{odd_asym}
x=\beta-\frac{\alpha_{1,2m+1}}{\alpha_{0,2m+1}}\frac{1}{\beta}-\left(\left(\frac{\alpha_{1,2m+1}}{\alpha_{0,2m+1}}\right)^2\frac{\alpha_{3,2m+1}}{\alpha_{0,2m+1}}-
\frac{\alpha_{1,2m+1}\alpha_{2,2m+1}}{(\alpha_{0,2m+1})^2}-\frac{1}{3}\frac{(\alpha_{1,2m+1})^3}{(\alpha_{0,2m+1})^3}\right)\frac{1}{\beta^3}-{\ldots}.
\end{equation}
Now, let us consider the even derivative of Bessel functions. By using \eqref{2n+1_der}, we write
\begin{equation}\label{even_exp}
\sqrt{\frac{1}{2}\pi x}J_\nu^{(2m)}(x)=\cos\left(x-\frac{\nu \pi}{2}-\frac{\pi}{4}\right)\tau_\nu^{(2m)}(x)+\sin\left(x-\frac{\nu \pi}{2}-\frac{\pi}{4}\right)\theta_\nu^{(2m)}(x),
\end{equation}
where $$\theta_\nu^{(2m)}(x)=\sum_{n=0}^{\infty}\frac{\alpha_{2n+1,2m}}{x^{2n+1}}=\frac{\alpha_{1,2m}}{x}+\frac{\alpha_{3,2m}}{x^3}+\frac{\alpha_{5,2m}}{x^5}+\ldots$$ and $$\tau_\nu^{(2m)}(x)=\sum_{n=0}^{\infty}\frac{\alpha_{2n,2m}}{x^{2n}}=\alpha_{0,2m}+\frac{\alpha_{2,2m}}{x^2}+\frac{\alpha_{4,2m}}{x^4}+{\ldots}.$$
If $x$ is a zero of $J_\nu^{(2m)}(x)$, then considering the series expansions of $\theta_\nu^{(2m)}(x)$ and $\tau_\nu^{(2m)}(x)$, and the fact that $|\sin x|\leq1$, the first approximation is given by
\begin{equation}\label{cos_eqn}
\cos\left(x-\frac{\nu \pi}{2}-\frac{\pi}{4}\right)+\mathcal{O}\left(\frac{1}{x}\right)=0.
\end{equation}
When $x$ is large, the left-hand side is dominated by the first term and therefore, the above equation implies that
$$x\sim k\pi-\frac{\pi}{2}-\frac{\nu \pi}{2}-\frac{\pi}{4},$$
for some large integer $k$. This is the first approximation to the root of the equation \eqref{sin_eqn}. In view of the fact that large $x$ is equivalent to the integer $k$, and in view of the equation \eqref{cos_eqn} we write that
\begin{align*}
x=k\pi+\frac{\nu \pi}{2}-\frac{\pi}{4}+\mathcal{O}\left(\frac{1}{x}\right)
=\alpha+\mathcal{O}\left(\frac{1}{k}\right),
\end{align*}
where $\alpha=k\pi+\frac{1}{2}\nu \pi-\frac{\pi}{4}$. Now, we write $$\sin\left(x-\frac{1}{2}\nu \pi-\frac{\pi}{4}\right)=\sin\left(x-\alpha+\alpha-\frac{1}{2}\nu \pi-\frac{\pi}{4}\right)=\sin\left(x-\alpha+k\pi-\frac{\pi}{2}\right)=(-1)^{k+1}\cos(x-\alpha).$$
Moreover, we have that
$$\cos(x-\frac{1}{2}\nu \pi-\frac{\pi}{4})=\cos(x-\alpha+k\pi-\frac{\pi}{2})=(-1)^{k}\sin(x-\alpha).$$
Consequently, from the above discussion, for $x$ to be the root of the equation \eqref{even_exp}, we write
\begin{align*}
\cos\left(x-\frac{\nu \pi}{2}-\frac{\pi}{4}\right)\tau_\nu^{(2m)}(x)+\sin\left(x-\frac{\nu \pi}{2}-\frac{\pi}{4}\right)\theta_\nu^{(2m)}(x)=0,
\end{align*}
which implies
\begin{align*}
\tan(x-\alpha)=\frac{\theta_\nu^{(2m)}(x)}{\tau_\nu^{(2m)}(x)}=\left.\sum_{n=0}^{\infty}\frac{\alpha_{2n+1,2m}}{x^{2n+1}}\right/\sum_{n=0}^{\infty}\frac{\alpha_{2n,2m}}{x^{2n}},
\end{align*}
or equivalently
\begin{align*}
x-\alpha&=\arctan\left[\left.\sum_{n=0}^{\infty}\frac{\alpha_{2n+1,2m}}{x^{2n+1}}\right/\sum_{n=0}^{\infty}\frac{\alpha_{2n,2m}}{x^{2n}}\right]
\end{align*}
By using the expansion of $\arctan(x)$ and the binomial expansion, we obtain that
\begin{align*}
x=\alpha+\frac{\alpha_{1,2m}}{\alpha_{0,2m}}\frac{1}{x}+\left(\frac{\alpha_{3,2m}}{\alpha_{0,2m}}-
\frac{\alpha_{1,2m}\alpha_{2,2m}}{(\alpha_{0,2m})^2}-\frac{1}{3}\frac{(\alpha_{1,2m})^3}{(\alpha_{0,2m})^3}\right)\frac{1}{x^3}-{\ldots}.
\end{align*}
Substituting $x=\alpha+\mathcal{O}\left(\frac{1}{k}\right)$ in the last equation, we obtain
\begin{align*}
x=\alpha+\frac{\alpha_{1,2m}}{\alpha_{0,2m}}\frac{1}{\alpha}+\mathcal{O}\left(\frac{1}{\alpha^3}\right).
\end{align*}
Further substitution gives
\begin{equation}\label{even_asym}
x=\alpha+\frac{\alpha_{1,2m}}{\alpha_{0,2m}}\frac{1}{\alpha}+\left(\frac{\alpha_{3,2m}}{\alpha_{0,2m}}-\frac{\alpha_{1,2m}\alpha_2^{(2m)}}{(\alpha_{0,2m})^2}-\frac{1}{3}\frac{(\alpha_{1,2m})^3}{(\alpha_{0,2m})^3}-\left(\frac{\alpha_{1,2m}}{\alpha_{0,2m}}\right)^2\right)\frac{1}{\alpha^3}-\cdots.
\end{equation}
Note that Baricz et. al. \cite[Theorem 1]{BKP18} proved that for $\nu\geq k$ all the zeros of $J_\nu^{(k)}(x)$ are positive and that the zeros of $n$-th and $(n+1)$-th derivatives of $J_\nu(x)$ are interlacing. For the zeros of $k$-th and $(k+1)$-th derivatives, we write
\begin{align*}
j_{\nu,n}^{(2m+1)}<j_{\nu,n}^{(2m)}<j_{\nu,n+1}^{(2m+1)}<j_{\nu,n+1}^{(2m)}<\cdots, \quad \nu\geq 2m.
\end{align*}
By using this interlacing property we conclude that \eqref{even_asym} and \eqref{odd_asym} represent the McMahon asymptotic expansion for $j_{\nu,k+1}^{(2m)}$ and $j_{\nu,k}^{(2m+1)}$, for $\nu\geq 2m$, respectively.

\begin{remark}
{\em We observe that the above asymptotic expansion is established for $\nu \geq 2m$. If the conjecture proposed by Baricz et al. \cite[Conjecture 1a]{BKP18} is proven true, this range could be significantly extended by using the above described procedure, thereby motivating further investigation into the validity of this conjecture and its implications for the asymptotic analysis of $J_\nu^{(n)}(x)$.}
\end{remark}

The error bound of an asymptotic expansion is important as it quantifies the accuracy of the approximation and ensures its reliability. It determines the range of validity and guides how many terms are required for a desired precision. Next, we are going to derive error bounds for the McMahon asymptotic expansion obtained above.

\subsection{An error bound}
In this section we find the bound for the approximation for the zeros of Bessel functions derivative by modifying the argument of Hethcote \cite{He70}, where he demonstrated that if $k\geq 0.314-\frac{1}{2}\nu+1.38|4\nu^2-1|,$ then
\begin{align*}
\left|j_{\nu,k}-\left(k+\frac{1}{2}\nu-\frac{1}{4}\right)\pi\right|\leq\frac{0.90|4\nu^2-1|}{\pi\left(k+\frac{1}{2}\pi-0.314\right)}.
\end{align*}
Here we show that if $\nu>2n-1$ and $k\geq -\frac{\nu}{2}+0.314+\frac{21}{\pi}\left((\nu+2n)^2-\frac{1}{4}\right),$ then
\begin{equation}\label{even_der_bnd}
\left|j_{\nu,k+1}^{(2n)}-\left(k+\frac{1}{2}\nu-\frac{1}{4}\right)\pi\right|\leq\frac{0.26(4(\nu+2n)^2-1)}{\pi\left(k+\frac{1}{2}\nu-0.314\right)}.
\end{equation}
Also for odd derivatives of Bessel functions we obtain a similar bound for the case when
if $\nu>2n$ and $k\geq -\frac{\nu}{2}+0.314+\frac{21}{\pi}\left((\nu+2n+1)^2-\frac{1}{4}\right)$
\begin{equation}\label{odd_der_bnd}
\left|j_{\nu,k}^{(2n+1)}-\left(k+\frac{1}{2}\nu-\frac{1}{4}\right)\pi\right|\leq\frac{0.26(4(\nu+2n+1)^2-1)}{\pi\left(k+\frac{1}{2}\nu-0.314\right)}.
\end{equation}
Let us first revisit the lemma presented in Hethcote \cite{He70}, which was derived using a method introduced by Gatteschi \cite{Ga56}.

\begin{lemma}\label{lemma1}
In the interval $[n\pi-\psi-\rho,n\pi-\psi+\rho]$, where $\rho<\frac{1}{2}\pi$, suppose $f(x)=\sin(x+\psi)+\epsilon(x)$, $f(x)$ is continuous and $E=\max|\epsilon(x)|<\sin\rho$. Then there exists a zero $c$ of $f(x)$ in the interval such that $|c-(n\pi-\psi)|\leq E\rho \csc \rho$.
\end{lemma}

From Theorem \ref{Theorem2}, we rewrite equation \eqref{2n_der_del}, which is the asymptotic expansion of even derivatives of Bessel functions as
\begin{align*}
(-1)^n\sqrt{\frac{2}{\pi x}}J_\nu^{(2n)}(x)= \sin\left(x-\frac{1}{2}\pi\nu+\frac{\pi}{4}\right)+(-1)^n\delta_{2n}(\nu,x),
\end{align*}
where we used $\cos(x-\frac{\pi}{2})=\sin x$. In order to use  Lemma  \ref{lemma1}, let us suppose that $$f(x)=(-1)^n\sqrt{\frac{2}{\pi x}}J_\nu^{(2n)}(x),\quad \psi=-\frac{1}{2}\nu \pi+\frac{1}{4}\pi,\quad \rho=0.2$$
and
\begin{equation}\label{epsl}
|\epsilon(x)|=\left|(-1)^n\delta_{2n}(\nu,x)\right|\leq \frac{4(\nu+2n)^2-1}{4x}\exp\left\{\frac{4(\nu+2n)^2-1}{4x}\right\}.
\end{equation}
Let $x\in [k\pi-\psi-\rho,k\pi-\psi+\rho]$. Then for the assumed value of $\rho \text{ and }\psi$ we obtain
\begin{equation}\label{x_lb1}
x\geq \pi\left(k+\frac{1}{2}\nu-0.314\right).
\end{equation}
We now use the result $0\leq y\leq \frac{l}{l+1}$ which implies that
\begin{align*}
y \exp y\leq \frac{y}{1-y}\leq l.
\end{align*}
Based on the above result, we conclude that
\begin{equation}\label{epsl_bnd}
\frac{4(\nu+2n)^2-1}{4x}\exp\left\{\frac{4(\nu+2n)^2-1}{4x}\right\}\leq\frac{1}{20}
\end{equation}
if
\begin{equation}\label{exp_bnd}
\frac{4(\nu+2n)^2-1}{4x}\leq \frac{1}{21},\ \text{i.e.}\ \frac{21(4(\nu+2n)^2-1)}{4}\leq x.
\end{equation}
Since $k$ is large, both \eqref{x_lb1} and \eqref{exp_bnd} hold if
\begin{align*}
\pi\left(k+\frac{1}{2}\nu-0.314\right)\geq\frac{21(4(\nu+2n)^2-1)}{4}.
\end{align*}
Solving the above inequality we obtain
\begin{equation*}\label{k_vlaue}
k\geq -\frac{\nu}{2}+0.314+\frac{21}{\pi}\left((\nu+2n)^2-\frac{1}{4}\right).
\end{equation*}
From \eqref{epsl}, \eqref{epsl_bnd} and Lemma \ref{lemma1}, we obtain
\begin{equation}\label{rln_E_epln}
E=\max|\epsilon(\nu,x)|\leq0.05,
\end{equation}
where the maximum is taken over all $x$ satisfying \eqref{x_lb1} and $\nu>2n-1$. From the above discussion it is clear that
\begin{align*}
E\leq \sin 0.2=0.19867.
\end{align*}
Hence, by \eqref{even_asym} and Lemma \ref{lemma1} we obtain
\begin{equation}\label{bnd_1}
\left|j^{(2n)}_{\nu,k+1}-\left(k\pi+\frac{1}{2}\nu\pi-\frac{1}{4}\pi\right)\right|\leq E\rho\csc \rho, \quad \rho=0.2.
\end{equation}
From \eqref{epsl}, \eqref{x_lb1} and \eqref{exp_bnd} we arrive at
\begin{align*}
|\epsilon(x)|&\leq \frac{4(\nu+2n)^2-1}{4x}\exp\left\{\frac{4(\nu+2n)^2-1}{4x}\right\}\\
&\leq \frac{4(\nu+2n)^2-1}{4\pi\left(k+\frac{1}{2}\nu-0.314\right)}\exp(1/21)\\
&\leq \frac{0.26(4(\nu+2n)^2-1)}{\pi\left(k+\frac{1}{2}\nu-0.314\right)}.
\end{align*}
By using the above bound for $\epsilon(x)$, \eqref{rln_E_epln} and \eqref{bnd_1}, for $\rho=0.2$, we obtain the required bound \eqref{even_der_bnd}. Similarly, we can prove the bound for the error of zeros of odd order derivatives of Bessel functions \eqref{odd_der_bnd}.

The next section is dedicated for the asymptotic expansion of the zeros of $n$-th order derivatives of Bessel functions provided $\nu\to\infty$.

\section{\bf Asymptotic expansions for large $\nu$}
\subsection{Some older results} First let us recall some basic results regarding the expression of Bessel functions and its derivatives for large $\nu$.
The uniform expansions, with respect to $x\in(0,\infty)$, for Bessel functions and its derivatives from Olver \cite[p. 338 and 342]{Ol54} (cf. \cite[sect. 4]{WL90}) are given by
\begin{equation}\label{Bsl_LNu}
J_\nu(\nu x)\sim \phi(\zeta)\left\{\frac{\airya\left(\nu^{\frac{2}{3}}\zeta\right)}{\nu ^\frac{1}{3}}\sum_{s=0}^{\infty}\frac{A_{s,0}(\zeta)}{\nu^{2s}}+\frac{\airya'\left(\nu^{\frac{2}{3}}\zeta\right)}{\nu ^\frac{5}{3}}\sum_{s=0}^{\infty}\frac{B_{s,0}(\zeta)}{\nu^{2s}}\right\}
\end{equation}
and
\begin{equation}\label{Bsl_der_LNu}
J_\nu^\prime(\nu x)\sim \psi(\zeta)\left\{\frac{\airya\left(\nu^{\frac{2}{3}}\zeta\right)}{\nu ^\frac{4}{3}}\sum_{s=0}^{\infty}\frac{A_{s,1}(\zeta)}{\nu^{2s}}+\frac{\airya'\left(\nu^{\frac{2}{3}}\zeta\right)}{\nu ^\frac{2}{3}}\sum_{s=0}^{\infty}\frac{B_{s,1}(\zeta)}{\nu^{2s}}\right\}
\end{equation}
as $\nu\to \infty$, respectively. Here, $\zeta$ and $x$ are related in a one-to-one manner by the equations
\begin{align*}
\zeta&=\left\{\frac{3}{2}\int_{x}^{1}\frac{(1-x^2)^\frac{1}{2}}{x}dx\right\}^{\frac{2}{3}}\\
&=\left\{\frac{3}{2}\ln\frac{1+(1-x^2)^\frac{1}{2}}{x}-\frac{3}{2}(1-x^2)^\frac{1}{2}\right\}^\frac{2}{3},\quad 0<x\leq1,
\end{align*}
and
\begin{align*}
\zeta=-\left\{\frac{3}{2}\int_{1}^{x}\frac{(x^2-1)^\frac{1}{2}}{x}dx\right\}=-\left\{\frac{3}{2}(x^2-1)^\frac{1}{2}-\frac{3}{2} \arcsec x\right\}^\frac{2}{3},\quad x\geq 1,
\end{align*}
where
\begin{equation}\label{phi_exp}
\phi(\zeta)=\left(\frac{4\zeta}{1-x^2}\right)^\frac{1}{4}=\left(-\frac{2}{x}\diff{x}{\zeta}\right)^\frac{1}{2},
\end{equation}
and $\psi(\zeta)=\frac{2}{x\phi(\zeta)}$. The coefficients $A_{s,0}(\zeta),B_{s,0}(\zeta),A_{s,1}(\zeta)\text{ and }B_{s,1}(\zeta)$ are analytic in a region containing the real axis and given by a set of recurrence relations. Particularly, $A_{0,0}(\zeta)=1,$ $B_{0,1}(\zeta)=1,$ and if $0\leq x\leq 1,$ then
$$B_{0,0}(\zeta)=-\frac{3\tau-5\tau^3}{324\zeta^\frac{1}{2}}-\frac{5}{48\zeta^2}$$
and
$$A_{1,0}(\zeta)=\frac{81\tau^2-462\tau^4+385\tau^6}{1152}-\frac{7(3\tau-5\tau^3)}{1152\zeta^\frac{2}{3}}-\frac{455}{4608\zeta^3},$$
where $\tau=(1-x^2)^{-\frac{1}{2}}.$ Moreover, if $1\leq x \leq\infty,$ then
$$B_{0,0}(\zeta)=-\frac{3\tau-5\tau^3}{324\zeta^\frac{1}{2}}-\frac{5}{48\zeta^2}$$
and
\begin{align*}
A_{1,0}(\zeta)=\frac{81\tau^2-462\tau^4+385\tau^6}{1152}-\frac{7(3\tau-5\tau^3)}{1152\zeta^\frac{2}{3}}-\frac{455}{4608\zeta^3},
\end{align*}
where $\tau=(x^2-1)^{-\frac{1}{2}},$ and
\begin{equation}\label{chi_bnd}
A_{0,1}(\zeta)=X(\zeta)+\zeta B_{0,0}(\zeta),\quad
X(\zeta)=\frac{\phi^\prime(\zeta)}{\phi(\zeta)}=\frac{4-z^2\{\phi(\zeta)\}^6}{16\zeta},
\end{equation}
see Olver \cite[p. 16]{Ol62} for more details.

Now, let us revisit other expansions for $J_\nu(\nu x)$ and $J_\nu'(\nu x)$, along with their error bounds, which are similar to \eqref{Bsl_LNu} and \eqref{Bsl_der_LNu}. These expansions were first derived by Olver (see \cite{Ol62,Ol64}) and later discussed by Wong and Lang \cite{WL91}. From \cite[eq. 2.14]{WL91}, we write
\begin{equation*}
J_\nu(\nu x)=\frac{1}{1+\delta_1}\frac{\phi(\zeta)}{\nu^{1/3}}[\airya(\nu^{2/3}\zeta)+\varepsilon_1(\nu,\zeta)],
\end{equation*}
where
\begin{align*}
|\delta_1|\leq \frac{0.217}{\nu} \quad\text{ if }\nu\geq10.
\end{align*}
Also, for negative $\zeta$,
\begin{equation}\label{bnd_epsilon}
	\left|\varepsilon_1(\nu,\zeta)\right|\leq\frac{0.2102}{\sqrt{\pi}(-\nu^{2/3})^{1/4}\nu}e^{0.03/\nu}.
\end{equation}
A uniform asymptotic approximation of $J_\nu^\prime(\nu x)$ (cf. \cite[sect. 5]{WL91}), similar to the above expression is given by
\begin{equation}\label{Bsl_der}
J^\prime_\nu(\nu x)=-\frac{1}{1+\delta_1}\frac{\psi(\zeta)}{\nu^{2/3}}\left[\frac{\airya(\nu^{2/3}\zeta)}{\nu^{2/3}}\left[C_0(\zeta)-\zeta B_0(\zeta)\right]+\airya'(\nu^{2/3}\zeta)+\eta_1(\nu,\zeta)+\chi(\zeta)\frac{\varepsilon_1(\nu,\zeta)}{\nu^{2/3}}\right],
\end{equation}
where $\phi(\zeta),$ $\psi(\zeta)$ and $\varepsilon_1(\nu,\zeta)$ are the same as in \eqref{phi_exp}, \eqref{Bsl_der_LNu} and \eqref{bnd_epsilon}, respectively. Moreover, $\chi(\zeta)$ and $C_0(\zeta)$  are given by \eqref{chi_bnd} and
\begin{align*}
\left|\eta_1(\nu,\zeta)\right|\leq\frac{0.2102}{\nu}e^{0.30/\nu}N(\nu^{2/3\zeta})
\end{align*}
for negative $\zeta$. Also, (see \cite[p. 750]{Ol63}) we arrive at
\begin{equation}\label{bnd_A'}
N(x)=\left[E^2(x)\airya'^2(x)+E^{-2}(x)\airyb'^2(x)\right]^{1/2}.
\end{equation}
For $x\leq-1$, the bound for $N(x)$ \cite[p. 752]{Ol63} is given by
\begin{equation}\label{N_bnd}
0<|x|^{-1/4}N(x)<0.06.
\end{equation}

\subsection{Expression for $J_\nu^{(n)}(x)$ and $J_\nu^{(n)}(x\nu)$}\label{J^n=J+J'}
We begin by deriving an expression for the $n$-th derivative of the Bessel functions in terms of its fist-order derivative. This representation facilitates the extension of known results for Bessel functions to their higher-order derivatives. For real values of $\nu$ and $x$, the Bessel differential equation is given by
\begin{align*}
x^2J_\nu^{\prime\prime}(x)+xJ_\nu^\prime(x)+(x^2-\nu^2)J_\nu(x)=0.
\end{align*}
For $x \neq 0$, we rewrite the above equation as follows:
\begin{equation}\label{eq: 2nd_der_Bsl}
J_\nu^{\prime\prime}(x)=-\frac{1}{x}J_\nu^\prime(x)+\left(\frac{\nu^2}{x^2}-1\right)J_\nu(x)=\beta_2(x,\nu)J_\nu^\prime(x)+\gamma_2(x,\nu)J_\nu(x),
\end{equation}
where
\begin{equation}\label{bg_2}
\beta_2(x,\nu)=-\frac{1}{x} \text{ and } \gamma_2(x,\nu)=\left(\frac{\nu^2}{x^2}-1\right).
\end{equation} 
Further differentiating  \eqref{eq: 2nd_der_Bsl} and using the same equation, we obtain
\begin{align*}
J_\nu^{(3)}(x)&=\beta_2^\prime(x,\nu)J_\nu^\prime(x)+\beta_2(x,\nu)J_\nu^{\prime\prime}(x)+\gamma_2^\prime(x,\nu)J_\nu(x)+\gamma_2(x,\nu)J_\nu^\prime(x)\\
&=(\beta_2^\prime(x,\nu)+\gamma_2(x,\nu))J_\nu^\prime(x)+\beta_2(x,\nu)(\beta_2(x,\nu)J_\nu^\prime(x)+\gamma_2(x,\nu)J_\nu(x))+\gamma_2^\prime(x,\nu)J_\nu(x)\\
&=(\beta_2^2(x,\nu)+\beta_2^\prime(x,\nu)+\gamma_2(x,\nu))J_\nu^\prime(x)+(\beta_2(x,\nu)\gamma_2(x,\nu)+\gamma_2^\prime(x,\nu))J_\nu(x)\\
&=\beta_3(x,\nu)J_\nu^\prime(x)+\gamma_3(x,\nu)J_\nu(x),
\end{align*}
where
\begin{align*}
\beta_3(x,\nu)=\beta_2^2(x,\nu)+\beta_2^\prime(x,\nu)+\gamma_2(x,\nu)\quad\text{ and }\quad\gamma_3(x,\nu)=\beta_2(x,\nu)\gamma_2(x,\nu)+\gamma_2^\prime(x,\nu).
\end{align*}
Proceeding in a similar manner and applying mathematical induction, we obtain the following expression for $n \geq 3$
\begin{equation}\label{n_der_Bsl}
J_\nu^{(n)}(x)=\beta_n(x,\nu)J_\nu^\prime(x)+\gamma_n(x,\nu)J_\nu(x),
\end{equation}
where
\begin{align*}
\beta_n(x,\nu)=\beta_{n-1}^2(x,\nu)+\beta_{n-1}^\prime(x,\nu)+\gamma_{n-1}(x,\nu)\quad\text{ and }\quad
\gamma_{n}(x,\nu)=\beta_{n-1}(x,\nu)\gamma_{n-1}(x,\nu)+\gamma_{n-1}^\prime(x,\nu)
\end{align*}
and $\beta_2$ and $\gamma_2$ are given by \eqref{bg_2}. It is interesting to note that \eqref{n_der_Bsl} can be written in terms of Bessel functions and its first and second derivatives. Now, let us consider the case when we replace $x$ by $\nu x$ in \eqref{eq: 2nd_der_Bsl}. We obtain that
\begin{equation}\label{J_2_der}
J_\nu^{(2)}(\nu x)=-\frac{1}{\nu x}J^\prime(\nu x)+\left(\frac{1}{x^2}-1\right)J_\nu(\nu x).
\end{equation}
Differentiating the above equation with respect to $x$, we deduce that
\begin{align*}
J_\nu^{(3)}(\nu x)&=\left(\frac{2}{\nu^2 x^2}+\left(\frac{1}{x^2}-1\right)\right)J^\prime(\nu x) -\left(\frac{2}{\nu x^3}+\frac{1}{\nu x}\left(\frac{1}{x^2}-1\right)\right)J_\nu(\nu x)\\
&=\left(\frac{1}{x^2}-1\right)J_\nu^\prime(\nu x)+\frac{2}{\nu^2x^2}J^\prime_\nu(\nu x)-\frac{2}{\nu x^3}J_\nu(\nu x)-\frac{1}{\nu x}\left(\frac{1}{x^2}-1\right)J_\nu(\nu x).
\end{align*}
For large $\nu$, the terms involving $\frac{1}{\nu}$ contribute to the error, while the leading term determines the principal part of the asymptotic expansion of $J_\nu^{(3)}(\nu x)$. Therefore, we rewrite the above equation as
\begin{align*}
J_\nu^{(3)}(\nu x)=\left(\frac{1}{x^2}-1\right)J_\nu^\prime(\nu x)+\frac{1}{\nu}\left[f_3\left(\frac{1}{x},\frac{1}{\nu}\right)J_\nu(\nu x)+g_3\left(\frac{1}{x},\frac{1}{\nu}\right)J_\nu^\prime(\nu x)\right],
\end{align*}
which, upon differentiating with respect to $x$, yields
\begin{align*}
\nu J_\nu^{(4)}(\nu x)=-\frac{2}{\nu x^3}J_\nu^\prime(\nu x)+\left(\frac{1}{x^2}-1\right)J_\nu^{\prime\prime}(\nu x)+\frac{1}{\nu}\left(\frac{1}{x^2}\right)\left[f_3^\prime\left(\frac{1}{x},\frac{1}{\nu}\right)J_\nu(\nu x)+g_3^\prime\left(\frac{1}{x},\frac{1}{\nu}\right)J_\nu^\prime(\nu x)\right],
\end{align*}
and this can be rewritten as
\begin{align*}
J_\nu^{(4)}(\nu x)=\left(\frac{1}{x^2}-1\right)^2J_\nu(\nu x)+\frac{1}{\nu}\left[f_4\left(\frac{1}{x},\frac{1}{\nu}\right)J_\nu(\nu x)+g_4\left(\frac{1}{x},\frac{1}{\nu}\right)J_\nu^\prime(\nu x)\right].
\end{align*}
We use mathematical induction to obtain the recurrence relation for the coefficients in the expression of $n$-th order derivative of Bessel functions as follows. 
For $n\geq1$, let us write the $(2n+1)$-th derivative of Bessel functions as
\begin{equation}\label{Besl_odd_der}
J_\nu^{(2n+1)}(\nu x)=\left(\frac{1}{x^2}-1\right)^nJ_\nu^\prime(\nu x)+\left[\mathcal{F}_{2n+1}\left(\frac{1}{x},\frac{1}{\nu}\right)J_\nu(\nu x)+\mathcal{G}_{2n+1}\left(\frac{1}{x},\frac{1}{\nu}\right)J_\nu^\prime(\nu x)\right],
\end{equation}
where $$\mathcal{F}_{2n+1}\left(\frac{1}{x},\frac{1}{\nu}\right)=\frac{1}{\nu}f_{2n+1}\left(\frac{1}{x},\frac{1}{\nu}\right)\ \mbox{and} \ \mathcal{G}_{2n+1}\left(\frac{1}{x},\frac{1}{\nu}\right)=\frac{1}{\nu}g_{2n+1}\left(\frac{1}{x},\frac{1}{\nu}\right).$$
Here $f_{2n+1}\left(\frac{1}{x},\frac{1}{\nu}\right) \text{ and }g_{2n+1}\left(\frac{1}{x},\frac{1}{\nu}\right)$ are polynomials in $\frac{1}{x}$. On differentiating both sides of \eqref{Besl_odd_der} with respect to $x$ we obtain that 
\begin{align*}
\nu J_\nu^{(2n+2)}(\nu x)=&-\frac{2n}{x^3}\left(\frac{1}{x^2}-1\right)^{n-1}J_\nu^\prime(\nu x)+\nu\left(\frac{1}{x^2}-1\right)^nJ_\nu^{\prime\prime}(\nu x)+\left(-\frac{1}{x^2}\right)\frac{\partial \mathcal{F}_{2n+1}\left(\frac{1}{x},\frac{1}{\nu}\right)}{\partial \left(\frac{1}{x}\right)}J_\nu(\nu x)\\
&+\nu \mathcal{F}_{2n+1}\left(\frac{1}{x},\frac{1}{\nu}\right)J_\nu^\prime(\nu x)-\frac{1}{x^2}\frac{\partial \mathcal{G}_{2n+1}\left(\frac{1}{x},\frac{1}{\nu}\right)}{\partial \left(\frac{1}{x}\right)}J_\nu^\prime(\nu x)+\nu J_\nu^{\prime\prime}(\nu x)\mathcal{G}_{2n+1}\left(\frac{1}{x},\frac{1}{\nu}\right).
\end{align*}
Upon dividing both sides of the above equation by $\nu$, substituting $J_\nu^{\prime\prime}(\nu x)$ from \eqref{J_2_der} and subsequently rearranging the terms, we obtain
\begin{align*}
J_\nu^{(2n+2)}(\nu x)&=\left(\frac{1}{x^2}-1\right)^{n+1}J_\nu(\nu x)+J_\nu(\nu x)\left[\mathcal{G}_{2n+1}\left(\frac{1}{x^2}-1\right)-\frac{1}{\nu x^2}\frac{\partial\mathcal{F}_{2n+1}\left(\frac{1}{x},\frac{1}{\nu}\right)}{\partial \left(\frac{1}{x}\right)}\right]\\
&+J_\nu^\prime(\nu x)\left[-\frac{1}{\nu x}\mathcal{G}_{2n+1}\left(\frac{1}{x},\frac{1}{\nu}\right)-\frac{2n}{\nu x^3}\left(\frac{1}{x^2}-1\right)^{n-1}+\mathcal{F}_{2n+1}\left(\frac{1}{x},\frac{1}{\nu}\right)-\frac{1}{\nu x^2}\frac{\partial \mathcal{G}_{2n+1}\left(\frac{1}{x},\frac{1}{\nu}\right)}{\partial\left(\frac{1}{x}\right)}\right]\\
&=\left(\frac{1}{x^2}-1\right)^{n+1}J_\nu(\nu x)+\mathcal{F}_{2n+2}\left(\frac{1}{x},\frac{1}{\nu}\right)J_\nu(\nu x)+\mathcal{G}_{2n+2}\left(\frac{1}{x},\frac{1}{\nu}\right)J_\nu^\prime(\nu x),
\end{align*}
where
$$\quad\mathcal{F}_{2n+2}\left(\frac{1}{x},\frac{1}{\nu}\right)=\mathcal{G}_{2n+1}\left(\frac{1}{x^2}-1\right)-\frac{1}{\nu x^2}\frac{\partial\mathcal{F}_{2n+1}\left(\frac{1}{x},\frac{1}{\nu}\right)}{\partial \left(\frac{1}{x}\right)}$$
and
$$\mathcal{G}_{2n+2}\left(\frac{1}{x},\frac{1}{\nu}\right)=-\frac{1}{\nu x}\mathcal{G}_{2n+1}\left(\frac{1}{x},\frac{1}{\nu}\right)-\frac{2n}{\nu x^3}\left(\frac{1}{x^2}-1\right)^{n-1}+\mathcal{F}_{2n+1}\left(\frac{1}{x},\frac{1}{\nu}\right)-\frac{1}{\nu x^2}\frac{\partial \mathcal{G}_{2n+1}\left(\frac{1}{x},\frac{1}{\nu}\right)}{\partial\left(\frac{1}{x}\right)}.$$
It is worth to note that by using the fact that $\mathcal{F}_{2n+1}\left(\frac{1}{x},\frac{1}{\nu}\right)=\frac{1}{\nu}f_{2n+1}\left(\frac{1}{x},\frac{1}{\nu}\right)$ and $\mathcal{G}_{2n+1}\left(\frac{1}{x},\frac{1}{\nu}\right)=\frac{1}{\nu}g_{2n+1}\left(\frac{1}{x},\frac{1}{\nu}\right)$ we can write
\begin{align*}
\mathcal{F}_{2n+2}\left(\frac{1}{x},\frac{1}{\nu}\right)=\frac{1}{\nu}f_{2n+2}\left(\frac{1}{x},\frac{1}{\nu}\right)\ \ \mbox{and}\ \ \mathcal{G}_{2n+2}\left(\frac{1}{x},\frac{1}{\nu}\right)=\frac{1}{\nu}g_{2n+2}\left(\frac{1}{x},\frac{1}{\nu}\right),
\end{align*}
where $f_{2n+2}\left(\frac{1}{x},\frac{1}{\nu}\right)$ and $g_{2n+2}\left(\frac{1}{x},\frac{1}{\nu}\right)$ are polynomials in $\frac{1}{x}$. Similar results can be also established for the arbitrary odd derivatives of Bessel functions. Let us write
\begin{align*}
J_\nu^{(2n)}(\nu x)=\left(\frac{1}{x^2}-1\right)^nJ_\nu(\nu x)+\left[\mathcal{F}_{2n}\left(\frac{1}{x},\frac{1}{\nu}\right)J_\nu(\nu x)+\mathcal{G}_{2n}\left(\frac{1}{x},\frac{1}{\nu}\right)J_\nu^\prime(\nu x)\right].
\end{align*}
Differentiating both sides of the above equation with respect to $x$ and substituting $J_\nu^{\prime\prime}(x)$ by using \eqref{J_2_der} and then rearranging the expression, gives
\begin{align*}
J_\nu^{(2n+1)}(\nu x)=\left(\frac{1}{x^2}-1\right)^{n}J_\nu^\prime(\nu x)+\mathcal{F}_{2n+1}\left(\frac{1}{x},\frac{1}{\nu}\right)J_\nu(\nu x)+\mathcal{G}_{2n+1}\left(\frac{1}{x},\frac{1}{\nu}\right)J_\nu^\prime(\nu x),
\end{align*}
where
$$\quad\mathcal{F}_{2n+1}\left(\frac{1}{x},\frac{1}{\nu}\right)=-\frac{2n}{\nu x^3}\left(\frac{1}{x^2}-1\right)^{n-1}+\left(\frac{1}{x^2}-1\right)\mathcal{G}_{2n}\left(\frac{1}{x},\frac{1}{\nu}\right)-\frac{1}{\nu x^2}\frac{\partial\mathcal{F}_{2n}\left(\frac{1}{x},\frac{1}{\nu}\right)}{\partial \left(\frac{1}{x}\right)}$$
and
$$\mathcal{G}_{2n+2}\left(\frac{1}{x},\frac{1}{\nu}\right)=-\frac{1}{\nu x}\mathcal{G}_{2n}\left(\frac{1}{x},\frac{1}{\nu}\right)+\mathcal{F}_{2n}\left(\frac{1}{x},\frac{1}{\nu}\right)-\frac{1}{\nu x^2}\frac{\partial \mathcal{G}_{2n}\left(\frac{1}{x},\frac{1}{\nu}\right)}{\partial\left(\frac{1}{x}\right)}.$$
Notice that $\mathcal{F}_{2n+1}\left(\frac{1}{x},\frac{1}{\nu}\right)$ and $\mathcal{G}_{2n+1}\left(\frac{1}{x},\frac{1}{\nu}\right)$ can be expressed as
\begin{align*}
\mathcal{F}_{2n+1}\left(\frac{1}{x},\frac{1}{\nu}\right)=\frac{1}{\nu}f_{2n+1}\left(\frac{1}{x},\frac{1}{\nu}\right)\ \ \mbox{and}\ \ \mathcal{G}_{2n+1}\left(\frac{1}{x},\frac{1}{\nu}\right)=\frac{1}{\nu}g_{2n+1}\left(\frac{1}{x},\frac{1}{\nu}\right),
\end{align*}
where $f\left(\frac{1}{x},\frac{1}{\nu}\right)$ and $g\left(\frac{1}{x},\frac{1}{\nu}\right)$ are polynomials in $\frac{1}{x}$. We would like to mention that we can find similar type of relations for the modified Bessel functions and their derivatives which would help us to find out new bounds for ratio of modified Bessel functions and derivatives by modifying the method used by N\.asell  in \cite{Na78}.

\subsection{Asymptotic behaviour of zeros of  $J_\nu^{(n)}(x) \text{ and } J_\nu^{(n)}(x\nu)$}\label{zeros_behav}
From the discussion given in subsection \ref{J^n=J+J'}, for $n\geq2$, let us write the $2n$-th derivative of the Bessel functions as
\begin{equation}\label{2n_der_LN}
J_\nu^{(2n)}(\nu x)=\left(\frac{1}{x^2}-1\right)^nJ_\nu (\nu x)+\frac{1}{\nu}\left[f_{2n}\left(\frac{1}{x},\frac{1}{\nu}\right)J_\nu(\nu x)+g_{2n}\left(\frac{1}{x},\frac{1}{\nu}\right)J_\nu^\prime(\nu x)\right].
\end{equation}
Differentiating the above equation and rearranging the terms, we obtain
\begin{equation}\label{2n1_der_LN}
J_\nu^{(2n+1)}(\nu x)=\left(\frac{1}{x^2}-1\right)^nJ_\nu^\prime (\nu x)+\frac{1}{\nu}\left[f_{2n+1}\left(\frac{1}{x},\frac{1}{\nu}\right)J_\nu(\nu x)+g_{2n+1}\left(\frac{1}{x},\frac{1}{\nu}\right)J_\nu^\prime(\nu x)\right].
\end{equation}
Similarly, by differentiating this equation and rearranging the terms, we obtain the corresponding formula for the even-order derivative
\begin{align*}
J_\nu^{(2n+2)}(\nu x)=\left(\frac{1}{x^2}-1\right)^{n+1}J_\nu (\nu x)+\frac{1}{\nu}\left[f_{2n+2}\left(\frac{1}{x},\frac{1}{\nu}\right)J_\nu(\nu x)+g_{2n+2}\left(\frac{1}{x},\frac{1}{\nu}\right)J_\nu^\prime(\nu x)\right].
\end{align*}
We will use equations \eqref{2n_der_LN} and \eqref{2n1_der_LN} to denote the expression for general even and odd derivatives of Bessel functions for further analysis. We note that, in this paper, we restrict our attention to the case $\zeta < 0$, which corresponds to the location of the zeros of the derivatives of the Bessel functions. In further discussion we will use equations \eqref{2n_der_LN} and \eqref{2n1_der_LN}. Using equation \eqref{phi_exp},  we obtain
\begin{align*}
\left(\frac{1-x^2}{x^2}\right)=\frac{4\zeta}{x^2(\phi(\zeta))^4}.
\end{align*}
After replacing $\left(\frac{1-x^2}{x^2}\right)$ in the leading term of \eqref{2n_der_LN} we have that
\begin{align*}
J_\nu^{(2n)}(\nu x)=\left(\frac{4\zeta}{x^2\phi(\zeta)^4}\right)^nJ_\nu (\nu x)+\frac{1}{\nu}\left[f_{2n}\left(\frac{1}{x},\frac{1}{\nu}\right)J_\nu(\nu x)+g_{2n}\left(\frac{1}{x},\frac{1}{\nu}\right)J_\nu^\prime(\nu x)\right].
\end{align*}
We write the above equation as
\begin{equation}\label{2n_der_exp}
J_\nu^{(2n)}(\nu x)=\left(\frac{4\zeta}{x^2\phi(\zeta)^4}\right)^n\left[J_\nu (\nu x)+\frac{1}{\nu}\left(\frac{x^2\phi(\zeta)^4}{4\zeta}\right)^n\left[f_{2n}\left(\frac{1}{x},\frac{1}{\nu}\right)J_\nu(\nu x)+g_{2n}\left(\frac{1}{x},\frac{1}{\nu}\right)J_\nu^\prime(\nu x)\right]\right].
\end{equation}
Note that $\phi(\zeta)$ is a non-negative increasing function on $(-\infty, 0]$ (see \cite{La89}). Consequently, for $\zeta < 0$, we have $0 \leq \phi(\zeta) \leq \phi(0) = 2^{1/3}$. In view of the assumption $n \ll \nu$ and the one-to-one correspondence between $x$ and $\zeta$, we conclude that for fixed $x$, the quantity $$\left(x^2\frac{\phi(\zeta)^4}{4\zeta}\right)^n$$ remains bounded for $\zeta < 0$. Moreover, it has been established in \cite[p. 10]{Ol62} (cf. \cite[§5]{WL91}) that $$\left|\frac{\phi^\prime(\zeta)}{\phi(\zeta)}\right|<0.160.$$
For further analysis let us write \eqref{2n_der_exp} as
 \begin{equation}\label{2n_der_exp2}
 	J_\nu^{(2n)}(\nu x)=\left(\frac{4\zeta}{x^2\phi
 		(\zeta)^4}\right)^n\left[J_\nu (\nu x)+\frac{1}{\nu^{2/3}}\frac{x^{2n}\phi(\zeta)^{4n}}{4^n\zeta^{n-1}}\left[f_{2n}\left(\frac{1}{x},\frac{1}{\nu}\right)\frac{J_\nu(\nu x)}{\nu^{1/3}\zeta}+g_{2n}\left(\frac{1}{x},\frac{1}{\nu}\right)\frac{J_\nu^\prime(\nu x)}{\nu^{1/3}\zeta}\right]\right].
 \end{equation}
 In order to simplify the equation \eqref{2n_der_exp2}, we analyze the expressions
 $$\frac{J_\nu(\nu x)}{\nu^{1/3}\zeta}\ \ \mbox{and}\ \ \frac{J_\nu^\prime(\nu x)}{\nu^{1/3}\zeta}.$$
 Using \eqref{Bsl_der}, we can write
\begin{align*}
	\frac{J^\prime_\nu(\nu x)}{\nu^{1/3}\zeta}&=-\frac{1}{1+\delta_1}\frac{\psi(\zeta)}{\nu\zeta}\left[\frac{\airya(\nu^{2/3}\zeta)}{\nu^{2/3}}\left\{C_0(\zeta)-\zeta B_0(\zeta)\right\}+\airya'(\nu^{2/3}\zeta)+\eta_1(\nu,\zeta)+\chi(\zeta)\frac{\varepsilon_1(\nu,\zeta)}{\nu^{2/3}}\right]\\
	&=-\frac{1}{1+\delta_1}\frac{\psi(\zeta)}{\nu^{1/3}}\left[\frac{\airya(\nu^{2/3}\zeta)}{\nu^{4/3}\zeta}\left\{C_0(\zeta)-\zeta B_0(\zeta)\right\}+\frac{\airya'(\nu^{2/3}\zeta)}{\nu^{2/3}\zeta}+\frac{\eta_1(\nu,\zeta)}{\nu^{2/3}\zeta}+\chi(\zeta)\frac{\varepsilon_1(\nu,\zeta)}{\nu^{4/3}\zeta}\right].
\end{align*}
By using \eqref{bnd_A'}, \eqref{N_bnd} and the fact that $\zeta<0,$ we obtain that
\begin{align*}
	\airya'(\nu^{2/3}\zeta)\leq\frac{ N(\nu^{2/3}\zeta)}{\nu^{2/3}\zeta}\leq0.06.
\end{align*}
Moreover, from \eqref{chi_bnd} and \eqref{bnd_epsilon} we conclude that $\frac{J^\prime_\nu(\nu x)}{\nu^{1/3}\zeta}$ is bounded, for specified $\zeta<0$ and large $\nu$. Similarly, we can show that the quantity $\frac{J_\nu(\nu x)}{\nu^{1/3}\zeta}$ in \eqref{2n_der_exp2} is also bounded. Based on the above discussion and \eqref{2n_der_exp2}, for large $\nu$, it follows that
\begin{equation}\label{2n_der_asymp}
J_\nu^{(2n)}(\nu x)=\frac{1}{(1+\delta_1)\nu^{1/3}}\left(\frac{4\zeta}{x^2\phi(\zeta)^4}\right)^n\left[\airya\left(\nu^{2/3}\zeta\right)+\delta(\nu,\zeta)\right],
\end{equation}
with
\begin{equation}\label{delta_bnd}
\left|\delta(\nu,\zeta)\right|\leq \frac{M}{\nu^\alpha},
\end{equation}
for some $\alpha>0$. This result is modified to the result obtained in \cite{WL90}. We now use a result of Hethcote \cite{He70}, which establishes a connection between the asymptotic behavior of transcendental functions and the asymptotic behavior of their zeros. In particular, we apply the following theorem \cite[Theorem 1]{He70} to derive the asymptotic expansion for $\zeta_{\nu, k}$.

\begin{theorem}\label{Heterto_trm}
In the interval $[a-\rho,a+\rho],$ suppose $f(\tau)=g(\tau)+\varepsilon(\tau)$ is continuous, $g(\tau)$ is differentiable, $g(a)=0, m=\min\left|g^\prime(\tau)\right|>0,$ and
\begin{equation}\label{max_e}
E=\max|\varepsilon(\tau)|<\min\{|g(a-\rho),|g(a+\rho)|\}.
\end{equation}
Then there exists a zero $c$ of $f(\tau)$ in the interval such that $|c-a|\leq E/m.$
\end{theorem}

Notice that, unlike the error bound in \cite[Sec. 5]{WL91}, the term $\delta_{2n}(\nu, \zeta)$ is not confirmed to be function of $\nu^{2/3}\zeta$, which limits the ability to predict the interaction behavior of $\nu\text{ and }\zeta$ in \eqref{2n_der_asymp}. To address this issue, we use a slightly modified version of the procedure followed by Wong and Lang \cite{WL91}. We would like to mention that while they applied the idea to the first nine negative zeros of the Airy function, our approach extends to all zeros, including those associated with the derivatives of Airy functions, by making use of the fact that $\nu$ is large. To apply Theorem \ref{Heterto_trm} in our context, we assume that $\tau=\zeta$. From \eqref{2n_der_asymp}, we write
\begin{align*}
\frac{J_\nu^{(2n)}(\nu x)(1+\delta_1)\nu^{1/3}x^{2n}(\phi(\zeta))^{4n}}{(4\zeta)^n}=\airya(\nu^{2/3}\zeta)+\delta(\nu,\zeta).
\end{align*}
After comparing the above expression with Theorem \ref{Heterto_trm}, we consider
\begin{align*}
f(\zeta)=\frac{J_\nu^{(2n)}(\nu x)(1+\delta_1)\nu^{1/3}x^{2n}(\phi(\zeta))^{4n}}{(4\zeta)^n},\ g(\zeta)=\airya\left(\nu^\frac{2}{3}\zeta\right)\ \ \mbox{and}\ \ \epsilon(\zeta)=\delta(\nu,\zeta).
\end{align*}
Also, let us denote the $k$-th negative root of $\airya(x)$ and $\airya'(x)$ by $a_k$ and $a_k'$, respectively. Then set $a=\frac{a_k}{\nu^{2/3}}$. Further let us choose $\rho_k$ such that
\begin{equation}\label{ineqAi}
\frac{a^\prime_{k+1}}{\nu^{2/3}}<\frac{a_k}{\nu^{2/3}}-\rho_k<\frac{a_k}{\nu^{2/3}}+\rho_k<\frac{a^\prime_k}{\nu^{2/3}}.
\end{equation}
The above inequality holds true due to the fact that $\airya(x)$ and $\airya'(x)$ have alternating zeros. Notice that if we replace $a$ and $\rho$ by $\frac{a_k}{\nu^{2/3}}$ and $\rho_k$ in Theorem \ref{Heterto_trm}, then
$$\left|g(a-\rho)\right|=\left|\airya(\nu^{2/3}(a-\rho))\right|=\left|\airya(a_k-\nu^{2/3}\rho_k)\right|>\max|\epsilon(\zeta)|.$$
The above inequality follows from the fact that for large $\nu$, from \eqref{ineqAi}, $\nu^{2/3}\rho_k$ is a number which satisfies the inequality
\begin{equation*}
0<\nu^{2/3}\rho_k<a^\prime_k-a_k,
\end{equation*}
while $\epsilon(\zeta)$ vanishes for large $\nu$.
Similarly, we can conclude that
$$\left|g(a+\rho)\right|>\max|\epsilon(\zeta)|,$$
which means that \eqref{max_e} is satisfied for the provided $\zeta$, namely $\zeta<0$. Another way to verify the inequality \eqref{max_e} under the given setup is by rewriting \eqref{ineqAi} as
\begin{align*}
\frac{a^\prime_{k+1}}{\nu^{2/3}}<\frac{a_k}{\nu^{2/3}}-\frac{\rho_k}{\nu^{2/3}}<\frac{a_k}{\nu^{2/3}}+\frac{\rho_k}{\nu^{2/3}}<\frac{a^\prime_k}{\nu^{2/3}}
\end{align*}
and subsequently estimating  $|g\left(a\pm \rho\right)|.$ In view of the inequality \eqref{ineqAi}, for large $\nu$, the inequality \eqref{max_e} holds with $a$ and $\rho$ replaced by $\frac{a_k}{\nu^{2/3}}$ and $\rho_k$. Moreover, let us denote
\begin{equation}\label{alph_bta}
\alpha_k=\frac{a_k}{\nu^{2/3}}-\rho_k\ \ \mbox{and}\ \ \beta_k=\frac{ a_k}{\nu^{2/3}}+\rho_k.
\end{equation}
Note that $\frac{a_{k+1}^\prime}{\nu^{2/3}}$ and $\frac{a_{k}^\prime}{\nu^{2/3}}$ are two consecutive zeros of $\airya'(\nu^{2/3}\zeta)$ and $\frac{ a_k}{\nu^{2/3}}$ is a critical point of $\airya'(\nu^{2/3}\zeta)$ in the interval $[a_{k+1}^\prime,a_k^\prime].$ Therefore, the minimum value of $\left|\airya'(\nu^{2/3}\zeta)\right|$ is attained at $\alpha_k$ and $\beta_k$ which are defined in \eqref{alph_bta}. Additionally, from \eqref{delta_bnd} we conclude that $\varepsilon(\nu^{2/3}\zeta)=\frac{c_k}{\nu^\alpha}$ for $\alpha>0.$ Since all the conditions of Theorem \ref{Heterto_trm} are satisfied by $f(\zeta),$ $\varepsilon(\zeta)$ and $g(\zeta),$ we conclude that
\begin{align*}
\left|\nu^{2/3}\zeta_{\nu,k}-a_k\right|\leq \frac{c_k}{\nu^\alpha}.
\end{align*}
In other words, for large $\nu$ we write
\begin{equation}\label{zeta_asym}
\zeta_{\nu,k}=\nu^{-\frac{3}{3}} a_k +\mathcal{O}\left(\frac{1}{\nu^\alpha}\right),
\end{equation}
for $\alpha>0$. 

Now, we use the above argument to derive the asymptotic behavior of the zeros of odd derivative of Bessel functions. After replacing $\left(\frac{1-x^2}{x^2}\right)$ in the leading term of \eqref{2n1_der_LN} we write
\begin{align*}
J_\nu^{(2n+1)}(\nu x)=\left(\frac{4\zeta}{x^2\phi(\zeta)^4}\right)^nJ_\nu^\prime (\nu x)+\frac{1}{\nu}\left[f_{2n+1}\left(\frac{1}{x},\frac{1}{\nu}\right)J_\nu(\nu x)+g_{2n+1}\left(\frac{1}{x},\frac{1}{\nu}\right)J_\nu^\prime(\nu x)\right].
\end{align*}
We may rewrite the above equation as
\begin{equation}\label{2n1_der_exp}
J_\nu^{(2n+1)}(\nu x)=\left(\frac{4\zeta}{x^2\phi(\zeta)^4}\right)^n\left[J_\nu^\prime (\nu x)+\frac{1}{\nu}\left(\frac{x^2\phi(\zeta)^4}{4\zeta}\right)^n\left[f_{2n+1}\left(\frac{1}{x},\frac{1}{\nu}\right)J_\nu(\nu x)+g_{2n+1}\left(\frac{1}{x},\frac{1}{\nu}\right)J_\nu^\prime(\nu x)\right]\right].
\end{equation}
For further analysis let us write \eqref{2n1_der_exp} as
\begin{equation}\label{2n1_der_exp2}
J_\nu^{(2n+1)}(\nu x)=\left(\frac{4\zeta}{x^2\phi(\zeta)^4}\right)^n\left[J_\nu^\prime (\nu x)+\frac{1}{\nu^{2/3}}\frac{x^{2n}\phi(\zeta)^{4n}}{4^n\zeta^{n-1}}\left[f_{2n+1}\left(\frac{1}{x},\frac{1}{\nu}\right)\frac{J_\nu(\nu x)}{\nu^{1/3}\zeta}+g_{2n+1}\left(\frac{1}{x},\frac{1}{\nu}\right)\frac{J_\nu^\prime(\nu x)}{\nu^{1/3}\zeta}\right]\right].
\end{equation}
Based on the above discussion, for large $\nu$, it follows that
\begin{equation}\label{2n1_der_asymp}
J_\nu^{(2n+1)}(\nu x)=\frac{1}{(1+\delta_1)\nu^{1/3}}\left(\frac{4\zeta}{x^2\phi(\zeta)^4}\right)^n\left[\airya'\left(\nu^{2/3}\zeta\right)+\delta_{2n+1}(\nu,\zeta)\right],
\end{equation}
with
\begin{equation}\label{delta1_bnd}
\left|\delta_{2n+1}(\nu,\zeta)\right|\leq \frac{M_{2n+1}}{\nu^\alpha},
\end{equation}
for some $\alpha>0$. Now we use Theorem \ref{Heterto_trm} to derive the asymptotic behavior of the zeros of $J_\nu^{2n+1}(\nu x)$.
In order to apply the theorem, let us assume that $\tau=\zeta$.
From \eqref{2n1_der_asymp}, we write
\begin{align*}
\frac{J_\nu^{(2n+1)}(\nu x)(1+\delta_1)\nu^{1/3}x^{2n}(\phi(\zeta))^{4n}}{(4\zeta)^n}=\airya'(\nu^{2/3}\zeta)+\delta_{2n+1}(\nu,\zeta).
\end{align*}
Moreover, replace $f(x)$ by $f_{2n+1}( \zeta)$ in Theorem \ref{Heterto_trm}. In order to use Theorem \ref{Heterto_trm}, we suppose
\begin{align*}
f_{2n+1}(\zeta)=\frac{J_\nu^{(2n+1)}(\nu x)(1+\delta_1)\nu^{1/3}x^{2n}(\phi(\zeta))^{4n}}{(4\zeta)^n}\ \ \mbox{and}\ \ g_{2n+1}(\zeta)=\airya'\left(\nu^\frac{2}{3}\zeta\right),
\end{align*}
and we choose $\rho_k$ such that
\begin{equation}\label{ineqAi'}
\frac{a_{k+1}}{\nu^{2/3}}<\frac{a_k^\prime}{\nu^{2/3}}-\rho_k<\frac{a_k^\prime}{\nu^{2/3}}+\rho_k<\frac{a_k}{\nu^{2/3}}.
\end{equation}
The above inequalities hold true due to the fact that $\airya(x)$ and $\airya'(x)$ have alternating zeros. In view of \eqref{ineqAi'}, for large $\nu$, the inequality \eqref{max_e} holds with $a$ and $\rho$ replaced by $\frac{a_k^\prime}{\nu^{2/3}}$ and $\rho_k$. Further let us denote
\begin{equation}\label{alph_bta1}
\alpha^\prime_k=\frac{a_k^\prime}{\nu^{2/3}}-\rho_k\ \ \mbox{and}\ \ \beta^\prime_k=\frac{ a_k^\prime}{\nu^{2/3}}+\rho_k.
\end{equation}
Note that $\frac{a_{k+1}^\prime}{\nu^{2/3}}\text{ and }\frac{a_{k}^\prime}{\nu^{2/3}}$ are two consecutive zeros of $\airya'(\nu^{2/3}\zeta)$ and $\frac{ a_k}{\nu^{2/3}}$ is a critical point of $\airya(\nu^{2/3}\zeta)$ in the interval $[a_{k+1},a_k]$. Therefore the minimum value of $\left|\airya'(\nu^{2/3}\zeta)\right|$ is attained at $\alpha_k^\prime$ and $\beta_k^\prime$ which are defined by \eqref{alph_bta1}. Additionally, from \eqref{delta1_bnd}, we conclude that $\varepsilon(\nu^{2/3\zeta})=\frac{c_k}{\nu^\alpha}$ for $\alpha>0$. Since all the conditions of Theorem \ref{Heterto_trm} are satisfied by $f_{2n+1}(\zeta),$ $\varepsilon(\zeta)$ and $g_{2n+1}(\zeta)$, we conclude that
\begin{align*}
\left|\nu^{2/3}\zeta_{\nu,k}^{(2n+1)}-a_k^\prime\right|\leq \frac{c_k}{\nu^\alpha}.
\end{align*}
In other words, for large $\nu$ we write
\begin{equation}\label{zeta_asym1}
\zeta_{\nu,k}^{(2n+1)}=\nu^{-\frac{2}{3}} a_k^\prime +\mathcal{O}\left(\frac{1}{\nu^\alpha}\right),
\end{equation}
for $\alpha>0$.

\subsection{Uniform asymptotic expansion for $J_\nu^{(n)}(\nu x)$}
We now use \eqref{Bsl_LNu} and \eqref{Bsl_der_LNu} to obtain the uniform expansion for the even derivatives of Bessel functions.
We apply mathematical induction to determine the expression for $J_\nu^{(n)}(x\nu)$ when $\nu$ is large. It is worth mentioning that we will use the procedure of Olver \cite{Ol54} to obtain the asymptotic expansion for higher-order derivatives of Bessel functions. Similar results can also be derived by expressing the $J_\nu^{(n)}(x)$ in terms of Bessel functions and its first derivative (see \eqref{2n_der_LN} and \eqref{2n1_der_LN}), and then substituting the expansions of $J_\nu(x)$ and $J_\nu^\prime(x)$. However, since the second method is very complicated, we adopt the first one.
For $n\leq 2m+1$, let us suppose that the odd derivatives of Bessel functions have the form given below
\begin{equation}\label{odd_jnux}
J_\nu^{(2n+1)}(\nu x)\sim -\psi_{2n+1}(\zeta)\left\{\frac{\airya\left(\nu^{\frac{2}{3}}\zeta\right)}{\nu ^\frac{4}{3}}\sum_{s=0}^{\infty}\frac{A_{s,2n+1}^{}(\zeta)}{\nu^{2s}}+\frac{\airya'\left(\nu^{\frac{2}{3}}\zeta\right)}{\nu ^\frac{2}{3}}\sum_{s=0}^{\infty}\frac{B_{s,2n+1}^{}(\zeta)}{\nu^{2s}}\right\},
\end{equation}
which align with the \eqref{Bsl_der_LNu} for $n=0$. Differentiating \eqref{odd_jnux} with respect to $\zeta$ we obtain
\begin{align*}
\nu J_\nu^{(2n+2)}(\nu x) \diff{x}{\zeta}&\sim -\left[-\frac{\psi_{2n+1}^\prime(\zeta)}{\psi_{2n+1}(\zeta)}J_\nu^{(2n+1)}(\nu x)+\psi_{2n+1}(\zeta)\left\{\frac{\airya\left(\nu^{\frac{2}{3}}\zeta\right)}{\nu ^\frac{4}{3}}\sum_{s=0}^{\infty}\frac{A_{s,2n+1}^{\prime}(\zeta)}{\nu^{2s}}\right. \right.\\ &\hspace*{2em}\left.\left.+\zeta\nu ^\frac{2}{3} \airya\left(\nu^{\frac{2}{3}}\zeta\right)\sum_{s=0}^{\infty}\frac{B_{s,2n+1}(\zeta)}{\nu^{2s}}+\frac{\airya'\left(\nu^{\frac{2}{3}}\zeta\right)}{\nu ^\frac{2}{3}}\sum_{s=0}^{\infty}\frac{A_{s,2n+1}(\zeta)+B_{s,2n+1}^\prime(\zeta)}{\nu^{2s}}\right\}\right],
\end{align*}
where we used the fact that $\airya''(x)=x\airya(x)$. The above asymptotic equation can be rewritten as
\begin{align*}
J_\nu^{(2n+2)}(\nu x) &\sim -\diff{\zeta}{x}\psi_{2n+1}(\zeta)\left[-\frac{\psi_{2n+1}^\prime(\zeta)}{\nu\psi_{2n+1}^2(\zeta)}J_\nu^{(2n+1)}(\nu x)+\left\{\frac{\airya\left(\nu^{\frac{2}{3}}\zeta\right)}{\nu ^\frac{1}{3}}\sum_{s=0}^{\infty}\frac{A_{s-1,2n+1}^{\prime}(\zeta)+\zeta B_{s,2n+1}(\zeta)}{\nu^{2s}}\right. \right.\\ &\hspace*{2em}\left.\left.+\frac{\airya'\left(\nu^{\frac{2}{3}}\zeta\right)}{\nu ^\frac{5}{3}}\sum_{s=0}^{\infty}\frac{A_{s,2n+1}(\zeta)+B_{s,2n+1}^\prime(\zeta)}{\nu^{2s}}\right\}\right].
\end{align*}
Moreover, by using \eqref{odd_jnux} we obtain that
\begin{align*}
J_\nu^{(2n+2)}(\nu x) &\sim -\diff{\zeta}{x}\psi_{2n+1}(\zeta)\left[\frac{\airya\left(\nu^{\frac{2}{3}}\zeta\right)}{\nu ^\frac{1}{3}}\sum_{s=0}^{\infty}\frac{\mbox{\Large$\chi$}_{2n+2}(\zeta)A_{s-1,2n+1}(\zeta)+A_{s-1,2n+1}^{\prime}(\zeta)+\zeta B_{s-1,2n+1}(\zeta)}{\nu^{2s}}\right. \\ &\hspace*{2em}\left.+\frac{\airya'\left(\nu^{\frac{2}{3}}\zeta\right)}{\nu ^\frac{5}{3}}\sum_{s=0}^{\infty}\frac{\mbox{\Large$\chi$}_{2n+2}(\zeta)B_{s-1,2n+1}(\zeta)+A_{s,2n+1}(\zeta)+B_{s-1,2n+1}^{\prime}(\zeta)}{\nu^{2s}}\right].
\end{align*}
By using \eqref{phi_exp}, we write the above equation as
\begin{equation}\label{2n+2_der}
J_\nu^{(2n+2)}(\nu x) \sim \psi_{2n+2}(\zeta)\left[\frac{\airya\left(\nu^{\frac{2}{3}}\zeta\right)}{\nu ^\frac{1}{3}}\sum_{s=0}^{\infty}\frac{A_{s-1,2n+2}(\zeta)}{\nu^{2s}}+\frac{\airya'\left(\nu^{\frac{2}{3}}\zeta\right)}{\nu ^\frac{5}{3}}\sum_{s=0}^{\infty}\frac{B_{s,2n+2}(\zeta)}{\nu^{2s}}\right],
\end{equation}
where
\begin{align*}
A_{s,2n+2}&=\mbox{\Large$\chi$}_{2n+2}(\zeta)A_{s-1,2n+1}(\zeta)+A_{s-1,2n+1}^{\prime}(\zeta)+\zeta B_{s-1,2n+1}(\zeta),\\
B_{s,2n+2}&=\mbox{\Large$\chi$}_{2n+2}(\zeta)B_{s-1,2n+1}(\zeta)+A_{s,2n+1}(\zeta)+B_{s-1,2n+1}^{\prime}(\zeta),\\
\psi_{2n+2}(\zeta)&=\frac{2\psi_{2n+1}(\zeta)}{x\phi^2(\zeta)}\quad\text{ and }\quad	\mbox{\Large$\chi$}_{2n+2}(\zeta)=\frac{\psi_{2n+1}^\prime(\zeta)}{\psi_{2n+1}(\zeta)}.
\end{align*}
We now derive the expansions for the odd derivatives of $J_\nu(\nu x)$, assuming that the even derivatives are given in the form stated above. To this end, we rewrite the above equation for the $2n$-th derivatives of the $J_\nu(\nu x)$ as
\begin{equation}\label{2n1_der_sim}
J_\nu^{(2n)}(\nu x)\sim \psi_{2n}(\zeta)\left[\frac{\airya\left(\nu^{\frac{2}{3}}\zeta\right)}{\nu ^\frac{1}{3}}\sum_{s=0}^{\infty}\frac{A_{s,2n}(\zeta)}{\nu^{2s}}+\frac{\airya'\left(\nu^{\frac{2}{3}}\zeta\right)}{\nu ^\frac{5}{3}}\sum_{s=0}^{\infty}\frac{B_{s,2n}(\zeta)}{\nu^{2s}}\right],
\end{equation}
where
\begin{align*}
\psi_{2n}(\zeta)=\frac{2\psi_{2n-1}(\zeta)}{x\phi^2(\zeta)}.
\end{align*}
Differentiating \eqref{2n1_der_sim} with respect with respect to $\zeta$ we obtain
\begin{align*}
\nu J_\nu^{(2n+1)}(\nu x) \diff{x}{\zeta}&\sim \left[\frac{\psi_{2n}^\prime(\zeta)}{\psi_{2n}(\zeta)}J_\nu^{(2n)}(\nu x)+\psi_{2n}(\zeta)\left\{\nu ^\frac{1}{3}\airya'\left(\nu^{\frac{2}{3}}\zeta\right)\sum_{s=0}^{\infty}\frac{A_{s,2n}(\zeta)}{\nu^{2s}}+\frac{\airya(\nu^{\frac{2}{3}}\zeta)}{\nu^\frac{1}{3}} \sum_{s=0}^{\infty}\frac{A_{s,2n}^\prime(\zeta)}{\nu^{2s}}\right. \right.\\ &\hspace*{2em}\left.\left.+\frac{\airya''\left(\nu^{\frac{2}{3}}\zeta\right)}{\nu}\sum_{s=0}^{\infty}\frac{B_{s,2n}(\zeta)}{\nu^{2s}}+
\frac{\airya'\left(\nu^{\frac{2}{3}}\zeta\right)}{\nu^{\frac{5}{3}}}\sum_{s=0}^{\infty}\frac{B_{s,2n}^\prime(\zeta)}{\nu^{2s}}\right\}\right].
\end{align*}
Rearranging the terms of the above equation and using the Airy differential equation, we conclude
\begin{align*}
J_\nu^{(2n+1)}(\nu x) &\sim \diff{\zeta}{x} \left[\frac{\psi_{2n}^\prime(\zeta)}{\nu\psi_{2n}(\zeta)}J_\nu^{(2n)}(\nu x)+\psi_{2n}(\zeta)\left\{\frac{\airya\left(\nu^{\frac{2}{3}}\zeta\right)}{\nu ^\frac{4}{3}}\sum_{s=0}^{\infty}\frac{A_{s,2n}^{\prime}(\zeta)+\zeta B_{s,2n}(\zeta)}{\nu^{2s}}\right. \right.\\ &\hspace*{2em}\left.\left.+\frac{\airya''\left(\nu^{\frac{2}{3}}\zeta\right)}{\nu ^\frac{2}{3}}\sum_{s=0}^{\infty}\frac{\nu^2A_{s,2n}(\zeta)+B_{s-1,2n}^\prime(\zeta)}{\nu^{2s}}\right\}\right].
\end{align*}
Moreover, we can write the above equation as
\begin{align*}
J_\nu^{(2n+1)}(\nu x) &\sim \diff{\zeta}{x}\psi_{2n}(\zeta)\left[\frac{\psi_{2n}^\prime(\zeta)}{\nu\psi_{2n}^2(\zeta)}J_\nu^{(2n)}(\nu x)+\left\{\frac{\airya\left(\nu^{\frac{2}{3}}\zeta\right)}{\nu ^\frac{4}{3}}\sum_{s=0}^{\infty}\frac{A_{s,2n}^{\prime}(\zeta)+\zeta B_{s,2n}(\zeta)}{\nu^{2s}}\right. \right.\\ &\hspace*{2em}\left.\left.+\frac{\airya'\left(\nu^{\frac{2}{3}}\zeta\right)}{\nu ^\frac{2}{3}}\sum_{s=0}^{\infty}\frac{\nu^2A_{s,2n}(\zeta)+B_{s-1,2n}^\prime(\zeta)}{\nu^{2s}}\right\}\right],
\end{align*}
which in turn, by using \eqref{2n1_der_sim}, implies that
\begin{align*}
J_\nu^{(2n+1)}(\nu x) &\sim \diff{\zeta}{x}\psi_{2n}(\zeta)\left[\frac{\airya\left(\nu^{\frac{2}{3}}\zeta\right)}{\nu ^\frac{4}{3}}\sum_{s=0}^{\infty}\frac{\mbox{\Large$\chi$}_{2n}(\zeta)A_{s,2n}(\zeta)+A_{s,2n}^{\prime}(\zeta)+\zeta B_{s,2n}(\zeta)}{\nu^{2s}}\right. \\ &\hspace*{2em}\left.+\frac{\airya'\left(\nu^{\frac{2}{3}}\zeta\right)}{\nu ^\frac{2}{3}}\sum_{s=0}^{\infty}\frac{A_{s,2n}(\zeta)+\mbox{\Large$\chi$}_{2n}(\zeta)B_{s-1,2n}(\zeta)+B_{s-1,2n}^\prime(\zeta)}{\nu^{2s}}\right].
\end{align*}
By using \eqref{phi_exp}, we can rewrite the above equation as
\begin{equation*}\label{2n1_der}
J_\nu^{(2n+1)}(\nu x) \sim -\psi_{2n+1}(\zeta)\left[\frac{\airya\left(\nu^{\frac{2}{3}}\zeta\right)}{\nu ^\frac{4}{3}}\sum_{s=0}^{\infty}\frac{A_{s,2n+1}(\zeta)}{\nu^{2s}}+\frac{\airya'\left(\nu^{\frac{2}{3}}\zeta\right)}{\nu ^\frac{2}{3}}\sum_{s=0}^{\infty}\frac{B_{s,2n+1}(\zeta)}{\nu^{2s}}\right],
\end{equation*}
where
\begin{align*}
A_{s,2n+1}&=\mbox{\Large$\chi$}_{2n}(\zeta)A_{s,2n}(\zeta)+A_{s,2n}^{\prime}(\zeta)+\zeta B_{s,2n}(\zeta),\\
B_{s,2n+1}&=A_{s,2n}(\zeta)+\mbox{\Large$\chi$}_{2n+1}(\zeta)B_{s-1,2n}(\zeta)+B_{s-1,2n}^\prime(\zeta),\\
\psi_{2n+1}(\zeta)&=\frac{2\psi_{2n}(\zeta)}{x\phi^2(\zeta)}\quad\text{ and }\quad	\mbox{\Large$\chi$}_{2n+1}(\zeta)=\frac{\psi_{2n}^\prime(\zeta)}{\psi_{2n}(\zeta)}	.
\end{align*}
By using mathematical induction we conclude the asymptotic form for odd and even derivatives of $J_\nu(\nu x)$.

\begin{remark}\label{xi_n_exp}
{\em We would like to mention that for $n\in\mathbb{N}_0$, we can express $\psi_n(\zeta)$ as
\begin{align*}
\psi_n(\zeta)=\frac{2^n}{x^n\phi^{2n-1}(\zeta)}.
\end{align*}
Also, we can write $\mbox{\Large$\chi$}_{n+1}(\zeta)$ as
\begin{equation}\label{xi_exp}
\mbox{\Large$\chi$}_{n+1}(\zeta)=\frac{\psi_n^\prime(\zeta)}{\psi_n(\zeta)}=-\frac{n}{x}-(2n-1)\frac{\phi^\prime(\zeta)}{\phi(\zeta)}.
\end{equation}
Therefore by knowing the expansion for $\frac{\phi^\prime(\zeta)}{\phi(\zeta)}$, we can write the asymptotic expansion of $\mbox{\Large$\chi$}_n(\zeta).$ Taking logarithm on both sides of \eqref{phi_exp}, we obtain that
\begin{equation}\label{log_phi}
\log(\phi(\zeta))=\frac{1}{4}\left(\log(4\zeta)-\log(1-x^2)\right).
\end{equation}
We now recall the relation between $x$ and $\zeta$, given by Olver \cite[p. 336]{Ol54},
\begin{equation}\label{x_zeta}
x(\zeta)=1-\frac{\zeta}{2^{\frac{1}{3}}}+\frac{3}{10}\frac{\zeta^2}{2^{\frac{2}{3}}}+\frac{1}{700}\zeta^3+{\ldots}.
\end{equation}
We use the above relation to expand the right hand side of \eqref{log_phi} in term of $\zeta$ as
\begin{align*}
\log(1-x^2)&=\log\left(1-\left(1-\frac{\zeta}{2^{\frac{1}{3}}}+\frac{2}{10}\frac{\zeta^2}{2^{\frac{2}{3}}}+\frac{1}{700}\zeta^3+\ldots\right)^2\right)\\
&= \log\zeta+\log a_0+\frac{a_1}{a_0}\zeta+\left(\frac{a_2}{a_0}-\frac{1}{2}\frac{a_1^2}{a_0^2}\right)\zeta^2+\ldots,
\end{align*}
where
\begin{align*}
a_0=2^{\frac{2}{3}},\ a_1=-\frac{2^{\frac{7}{3}}}{5}\ \mbox{and}\ a_2= \frac{52}{175}.
\end{align*}
Differentiating \eqref{log_phi} with respect to $\zeta$ and then  by using the above expansion for $\log(1-x^2)$, we obtain
\begin{equation}\label{phi_log_der}
\frac{\phi^\prime(\zeta)}{\phi(\zeta)}=\frac{1}{4}\left(-\frac{a_1}{a_0}-2\zeta\left(\frac{a_2}{a_0}-\frac{1}{2}\frac{a_1^2}{a_0^2}\right)+\ldots\right).
\end{equation}
Now, we write the expansion for $\mbox{\Large$\chi$}_{n+1}(\zeta)$ by using the expansion of $\frac{1}{x(\zeta)}$ in term of $\zeta$ as follows
\begin{align*}
\frac{1}{x(\zeta)}&=\frac{1}{1-\left(\frac{\zeta}{2^{1/3}}-\frac{3}{10}\frac{\zeta^2}{2^{2/3}}-\frac{1}{700}\zeta^3-\ldots\right)}\\
&=1+\frac{1}{2^{1/3}}\zeta+\frac{7}{10}\frac{1}{2^{2/3}}\zeta^2+\frac{139}{700}\zeta^3+{\ldots}.
\end{align*}
By using \eqref{xi_exp}, \eqref{phi_log_der} and above expansion, we obtain that
\begin{align*}
\mbox{\Large$\chi$}_{n+1}(\zeta)&=-\frac{n}{x}-(2n-1)\frac{\phi^\prime(\zeta)}{\phi(\zeta)}\\
&=\left(-n-\frac{n}{2^{1/3}}\zeta-\frac{7n}{10}\frac{1}{2^{2/3}}\zeta^2-\frac{139n}{700}\zeta^3-\ldots\right)-
\frac{1}{4}\left(-\frac{(2n-1)a_1}{a_0}-2(2n-1)\left(\frac{a_2}{a_0}-\frac{1}{2}\frac{a_1^2}{a_0^2}\right)\zeta+\ldots\right)\\
&=\frac{(2n-1)a_1}{4a_0}-n+\left(\frac{(2n-1)}{2}\left(\frac{a_2}{a_0}-\frac{1}{2}\frac{a_1^2}{a_0^2}\right)-\frac{n}{2^{1/3}}\right)\zeta+\ldots\\
&=\chi_{0,n}+\chi_{1,n}\zeta+\ldots,
\end{align*}
where 
$$\chi_{0,n}=\frac{(2n-1)a_1}{4a_0}-n\ \ \mbox{and}\ \ \chi_{1,n}=\frac{(2n-1)}{2}\left(\frac{a_2}{a_0}-\frac{1}{2}\frac{a_1^2}{a_0^2}\right)-\frac{n}{2^{1/3}}.$$}
\end{remark}

\subsection{Asymptotic expansion of zeros of $J_\nu^{(n)}(\nu x)$} In order to obtain the final result, we follow the procedure of Olver \cite{Ol54}, which was also used by Wong and Lang \cite{WL90}. In view of \eqref{2n1_der_sim}, we assume that
\begin{equation}\label{W_zeta}
W(\zeta)=\frac{\nu^{1/3}J_\nu^{(2n)}(\nu x)}{\psi_{2n}(\zeta)}=\airya\left(\nu^{\frac{2}{3}}\zeta\right)\sum_{s=0}^{\infty}\frac{A_{s,2n}(\zeta)}{\nu^{2s}}+
\frac{\airya'\left(\nu^{\frac{2}{3}}\zeta\right)}{\nu ^\frac{4}{3}}\sum_{s=0}^{\infty}\frac{B_{s,2n}(\zeta)}{\nu^{2s}}.
\end{equation}
By using \eqref{zeta_asym}, we write
\begin{align*}
\zeta_{\nu, k}=\nu^{-\frac{2}{3}}a_k+\eta_k,
\end{align*}
where $\eta_k=\mathcal{O}\left(\frac{1}{\nu^\alpha}\right)$. To avoid unnecessary complexity, let us denote
\begin{equation}\label{beta_epsilon}
\beta=\nu^{-\frac{2}{3}}a_k\ and\ \epsilon=\eta_k.
\end{equation}
Given the one-to-one relation between $x$ and $\zeta,$ and in view of \eqref{W_zeta}, we have that $W(\beta+\epsilon)=0,$ which can be written as
\begin{equation}\label{main_eqn_W}
W(\beta)+\frac{\epsilon}{1!}W^\prime(\beta)+\frac{\epsilon^2}{2!}W^{\prime\prime}(\beta)+\ldots=0.
\end{equation}
These higher order derivatives of $W(\beta)$ can be calculated from \eqref{W_zeta} as follows. Observe that \eqref{W_zeta} has the same structure as the first expansion in \cite[eq. 7.6]{Ol54} with $m=0.$ Therefore, by analogous reasoning, it follows that
\begin{equation}\label{W_der}
\left\{
\begin{aligned}
W^{(2m)}(\zeta) &\sim \nu^{2m}\airya\left(\nu^{\frac{ 2}{3}}\zeta\right)\sum_{r=0}^{\infty}\frac{A_{r,2n}^{2m}(\zeta)}{\nu^{2r}}+
\nu^{2m-\frac{4}{3}}\airya'\left(\nu^{\frac{2}{3}}\zeta\right)\sum_{r=0}^{\infty}\frac{B_r^{2m,2n}(\zeta)}{\nu^{2r}} \\
W^{(2m+1)}(\zeta) &\sim \nu^{2m}\airya\left(\nu^{\frac{ 2}{3}}\zeta\right)\sum_{r=0}^{\infty}\frac{A_{r,2n}^{2m}(\zeta)}{\nu^{2r}}+
\nu^{2m-\frac{4}{3}}\airya'\left(\nu^{\frac{2}{3}}\zeta\right)\sum_{r=0}^{\infty}\frac{B_r^{2m,2n}(\zeta)}{\nu^{2r}}
\end{aligned},
\right.
\end{equation}
where $A_{k,2n}^0=A_{k,2n},$ $B_{k,2n}^{0}=B_{k,2n}$ and
\begin{equation}\label{EF_der}
\left\{
\begin{aligned}
A_{r,2n}^{2m}&=\frac{d}{d\zeta}A_{r-1,2n}^{2m-1}+\zeta B_r^{2m-1,2n},\quad A_r^{2m+1,2n}=\frac{d}{d\zeta}A_{r,2n}^{2m}+\zeta B_{r,2n}^{2m,2n},\\
B_r^{2m,2n}&= A_{r,2n}^{2m-1}+\frac{d}{d\zeta}B_{r,2n}^{2m-1},\quad B_r^{2m+1}= A_r^{2m,2n}+\frac{d}{d\zeta}B_{r-1}^{2m},\\
\end{aligned}
\right.
\end{equation}
with $A_{-1,2n}^{2m-1}=B_{-1,2n}^{2m}=0$. Furthermore, note that $\beta$, see \eqref{beta_epsilon}, act as zero of $\airya\left(\nu^{\frac{2}{3}}\zeta\right)$. In order to obtain $W^{(2m)}(\beta)\text{ and }W^{(2m+1)}(\beta)$, we replace $\zeta$ by $\beta$ in \eqref{W_der}, that is
\begin{equation}\label{W_beta}
\left\{
\begin{aligned}
W^{(2m)}(\beta) &\sim \nu^{2m-\frac{4}{3}}\airya'\left(a_k\right)\sum_{r=0}^{\infty}\frac{B_{r,2n}^{2m}(\beta)}{\nu^{2r}}, \\
W^{(2m+1)}(\beta) &\sim \nu^{2m+\frac{2}{3}}\airya'\left(a_k\right)\sum_{r=0}^{\infty}\frac{B_{r,2n}^{2m+1}(\beta)}{\nu^{2r}}.
\end{aligned}
\right.
\end{equation}
Let us denote
\begin{equation}\label{f_exp}
\left\{
\begin{aligned}
f_{2m}(\beta)&=\frac{W^{(2m)}(\beta)}{\nu^{2m-\frac{4}{3}}\airya'(a_k)} \sim \sum_{r=0}^{\infty}\frac{B_{r,2n}^{2m}(\beta)}{\nu^{2r}},\\
f_{2m+1}(\beta)&=\frac{W^{(2m+1)}(\beta)}{\nu^{2m+\frac{2}{3}}\airya'(a_k)}\sim \sum_{r=0}^{\infty}\frac{B_{r,2n}^{2m+1}(\beta)}{\nu^{2r}},
\end{aligned}
\right.
\end{equation}
for large $\nu$. In other words, for $l\in\{0,1,2,\ldots\}$, we can write
\begin{equation}\label{f_F_rel}
f_l(\beta)\sim B_{0,2n}^l(\beta)+\frac{B_{1,2n}^l}{\nu^2}+\frac{B_{2,2n}^l(\beta)}{\nu^4}+\ldots, \quad \text{as} \nu \to \infty.
\end{equation}
Additionally, \eqref{f_exp} can be used to express the derivative of $W(\beta)$ in terms of $f_l(\beta)$ as follows
\begin{equation}\label{w_exp}
\left\{
\begin{aligned}
W^{(2m)}(\beta)&=f_{2m}(\beta)\nu^{2m-\frac{4}{3}}\airya'(a_k),\\
W^{(2m+1)}(\beta)&=f_{2m+1}(\beta)\nu^{2m+\frac{2}{3}}\airya'(a_k),
\end{aligned}
\right.
\end{equation}
Substituting \eqref{w_exp} in \eqref{main_eqn_W}, we obtain
\begin{equation*}\label{key}
f_0+\frac{\epsilon \nu^2}{1!}f_1+\frac{\epsilon^2\nu^2}{2!}f_2+\frac{\epsilon^3\nu^4}{3!}f_3+\ldots=0.
\end{equation*}
From the above equation we can write
\begin{equation}\label{epsln_kapa}
\epsilon\sim \frac{\kappa_1}{\nu^2}+\frac{\kappa_2}{\nu^4}+\frac{\kappa_3}{\nu^6}+\ldots,
\end{equation}
where $\kappa_1,\kappa_2,\ldots,$ depend on $\nu,$ and the first two coefficients on the right are given by
\begin{equation}\label{kappa_exp}
\kappa_1=-\frac{f_0}{f_1}\ \ \mbox{and}\ \ \kappa_2=-\frac{f_2}{2f_1}\kappa_1^2-\frac{f_3}{6f_1}\kappa_1^3.
\end{equation}
We now express $\kappa_1,\kappa_2,\ldots$ in terms of $f_i's$, by using \eqref{kappa_exp} and \eqref{f_F_rel}, as follows
\begin{align*}
\kappa_1=-\frac{f_0}{f_1}&=-\frac{B_{0,2n}^0(\beta)+\frac{B_{1,2n}^0}{\nu^2}+\frac{B_{2,2n}^0(\beta)}{\nu^4}+\ldots}
{B_{0,2n}^1(\beta)+\frac{B_{1,2n}^1}{\nu^2}+\frac{B_{2,2n}^1(\beta)}{\nu^4}+\ldots}\\
&=-\frac{B_{0,2n}^0(\beta)+\frac{B_{1,2n}^0(\beta)}{\nu^2}+\frac{B_{2,2n}^0(\beta)}{\nu^4}+\ldots}
{B_{0,2n}^1(\beta)\left(1+\frac{1}{\nu^2}\frac{B_{1,2n}^1(\beta)}{B_{0,2n}^1(\beta)}+\frac{1}{\nu^4}\frac{B_{2,2n}^1(\beta)}{B_{0,2n}^1(\beta)}+\ldots\right)}\\
&=-\frac{1}{B_{0,2n}^1(\beta)}\left(B_{0,2n}^0(\beta)+\frac{B_{1,2n}^0(\beta)}{\nu^2}+\frac{B_{2,2n}^0(\beta)}{\nu^4}+\ldots\right)\\
&\hspace*{5em}\left(1-\left(\frac{1}{\nu^2}\frac{B_{1,2n}^1(\beta)}{B_{0,2n}^1(\beta)}+\frac{1}{\nu^4}\frac{B_{2,2n}^1(\beta)}{B_{0,2n}^1(\beta)}+\ldots\right)+\ldots\right)\\
&=\frac{B_{0,2n}^0(\beta)}{B_{0,2n}^1(\beta)}+\frac{1}{\nu^2B_{0,2n}^1(\beta)}\left(B_{1,2n}^0(\beta)-\frac{B_{1,2n}^1(\beta)B_{0,2n}^0(\beta)}{B_{0,2n}^1(\beta)}\right)+\ldots\\	
\end{align*}
as $\nu\to \infty$.
Similarly, we can expand the expressions $\kappa_2, \kappa_3, \ldots$ in \eqref{kappa_exp} to obtain series containing $B_i$. Furthermore, substituting all these expressions in \eqref{epsln_kapa} and arranging the terms we obtain that
\begin{equation}\label{epsilon_exp}
\epsilon\sim \frac{\epsilon_1}{\nu^2}+\frac{\epsilon_2}{\nu^4}+\frac{\epsilon_3}{\nu^6}+\ldots,
\end{equation}
where
\begin{equation*}\label{key}
\epsilon_1=\frac{B_{0,2n}^0(\beta)}{B_{0,2n}^1(\beta)},\ \ \epsilon_2=-\frac{1}{B_{0,2n}^1(\beta)}\left(B_{1,2n}(\beta)+B_{1,2n}^1(\beta)\epsilon_1+\frac{1}{2}B_{0,2n}^2(\beta)\epsilon_1^2+\frac{1}{6}B_{0,2n}^3(\beta)\epsilon_1^3\right), \ \ {\ldots}.
\end{equation*}
Using the fact that $x$ and $\zeta$ are related one-to-one and in view of \eqref{W_zeta}, for fixed $\zeta_{\nu k}$, which is zero of $W(\zeta)$ in \eqref{W_zeta}, we can write
\begin{align*}
j_{\nu,k}^{(2n)}=\nu x_{\nu k}=\nu x(\zeta_{\nu k})=\nu x(\beta+\epsilon),
\end{align*}
by using \eqref{beta_epsilon}. Similar to \eqref{main_eqn_W}, by using the above relation we obtain that 
\begin{equation}\label{2n_x_exp}
j_{\nu,k}^{(2n)}=\nu x(\beta)+\frac{1}{1!}\nu \epsilon x^\prime(\beta)+\frac{1}{2!}\nu \epsilon^2x^{\prime\prime}(\beta)+\ldots.
\end{equation}
Here the expansion of $x(\zeta)$ is given by \eqref{x_zeta}. Moreover, differentiating the above equation with respect to $\zeta$ and substituting the derivatives in \eqref{2n_x_exp} we obtain
\begin{align*}
j_{\nu,k}^{(2n)}=\nu \left(1-\frac{\beta}{2^{1/3}}+\frac{2}{10}\frac{\beta^2}{2^{2/3}}+\frac{1}{700}\beta\right)+\nu \epsilon \left(-\frac{1}{2^{1/3}}+\frac{6}{10}\frac{1}{2^{2/3}}\beta+\frac{3}{700}\beta^2\right)+\ldots,
\end{align*}
where $\beta$ and $\epsilon$ are given by \eqref{beta_epsilon} and \eqref{epsilon_exp}.

The procedure to derive asymptotic expansion for zeros of $J_\nu^{(2n+1)}(x)$ is very similar to the case of even derivatives of Bessel functions. Let us assume
\begin{equation}\label{W1_zeta}
W_{2n+1}(\zeta)=\frac{\nu^{2/3}J_\nu^{(2n+1)}(\nu x)}{\psi_{2n+1}(\zeta)}=\airya'\left(\nu^{\frac{2}{3}}\zeta\right)\sum_{s=0}^{\infty}\frac{B_{s,2n}(\zeta)}{\nu^{2s}}+
\frac{\airya\left(\nu^{\frac{2}{3}}\zeta\right)}{\nu ^\frac{2}{3}}\sum_{s=0}^{\infty}\frac{A_{s,2n}(\zeta)}{\nu^{2s}}.
\end{equation}
From \eqref{zeta_asym1} we write
\begin{align*}
\zeta_{\nu, k}^{(2n+1)}=\nu^{-\frac{2}{3}}a_k^\prime+\eta_k,
\end{align*}
where $\eta_k=\mathcal{O}\left(\frac{1}{\nu^\alpha}\right)$. To avoid unnecessary complexity, let us denote
\begin{equation*}\label{beta_epsilon1}
\beta_1=\nu^{-\frac{2}{3}}a_k^\prime\text{ and }\epsilon_1=\eta_k.
\end{equation*}
In view of the one-to-one relation between $x$ and $\zeta,$ and by using \eqref{W_zeta}, we arrive at $W_1(\beta+\epsilon_1)=0,$ which can be written as
\begin{equation*}\label{main_eqn_W1}
	W_1(\beta_1)+\frac{\epsilon_1}{1!}W_1^\prime(\beta_1)+\frac{\epsilon_1^2}{2!}W_1^{\prime\prime}(\beta_1)+\ldots=0.
\end{equation*}
The rest of the steps are very similar to the case of zero of even derivatives of Bessel functions, so we omit the details.

In particular, the asymptotic expansion for zeros of third and fourth derivatives of Bessel functions is given by
\begin{align*}
j_{\nu,k}^{(3)}&=\nu \left(1-\frac{\beta}{2^{1/3}}+\frac{2}{10}\frac{\beta^2}{2^{2/3}}+\frac{1}{700}\beta\right)+\nu \epsilon \left(-\frac{1}{2^{1/3}}+\frac{6}{10}\frac{1}{2^{2/3}}\beta+\frac{3}{700}\beta^2\right)+\ldots\\
&=\nu-\frac{a_k}{2^{1/3}}\nu^{1/3}+\frac{2^{1/3}a_k^2}{10}\nu^{-1/3}+\ldots
\end{align*}
and
\begin{align*}
j_{\nu,k}^{(4)}&=\nu \left(1-\frac{\beta_1}{2^{1/3}}+\frac{2}{10}\frac{\beta_1^2}{2^{2/3}}+\frac{1}{700}\beta_1\right)+\nu \epsilon_1 \left(-\frac{1}{2^{1/3}}+\frac{6}{10}\frac{1}{2^{2/3}}\beta_1+\frac{3}{700}\beta_1^2\right)+\ldots\\
&=\nu-\frac{a_k^\prime}{2^{1/3}}\nu^{1/3}+\frac{2^{1/3}(a_k^\prime)^2}{10}\nu^{-1/3}+\frac{(a_k^\prime)^3}{700}\nu^{-1}+\cdots.
\end{align*}

We note that, unlike the case of the asymptotic expansion for the zeros of the second derivative of Bessel functions, the effect of
$\epsilon$ in the asymptotic expansion of $j_{\nu,k}^{(3)}$ appears only after the fourth term.

\section{\bf Discussion and future work}

In this paper, we extended several known results on higher-order derivatives of Bessel functions and their zeros and we obtained two different asymptotic expansions for the zeros $j_{\nu,k}^{(n)}$ of $J_\nu^{(n)}(x)$. More precisely, we established a McMahon-type expansion for the case when $k \to \infty$ with fixed $\nu$, including an explicit error bound, and another expansion for the case when $\nu \to \infty$ with fixed $k$. Our proofs were based on Hethcote's methods \cite{He70}, combined with general properties of the zeros of higher-order derivatives. The techniques used in this paper, together with several auxiliary generalizations, are expected to be useful in addressing a variety of related problems. For instance, the expression for the $n$-th derivative of Bessel functions obtained in section \ref{J^n=J+J'} can be applied to extend the results of Wong and Lee \cite{WL91}. Moreover, the same approach in section \ref{J^n=J+J'} can be adapted to derive analogous expressions for the derivatives of modified Bessel functions. Further, by modifying the method of N\.asell \cite{Na78} and using these derivative formulas, one can obtain rational bounds for ratios of modified Bessel functions and their derivatives.

\end{document}